\documentclass{scrartcl}

\usepackage{amsmath,amsfonts,amsthm, authblk}%
\usepackage[linesnumbered,vlined,boxed]{algorithm2e}%
\let\oldnl\nl
\newcommand{\nonl}{\renewcommand{\nl}{\let\nl\oldnl}}

\usepackage{xcolor}
\newcommand{\adrien}[1]{{\textcolor{red}{#1}}}
\renewcommand{\adrien}[1]{}
\newcommand{\martin}[1]{{\textcolor{blue}{#1}}}
\renewcommand{\martin}[1]{}

\usepackage{enumerate}%
\usepackage{tikz}%
\usepackage{pgfplots}%

\usepackage[numbers,sort&compress]{natbib}

\usepackage{hyperref}%
\hypersetup{
  colorlinks=true,
  breaklinks=true,
  urlcolor= blue,
  linkcolor= blue,
  bookmarksopen=true,
  pdftitle={Irreducibility test},
  pdfauthor={Martin Weimann},
}

\newcommand{\hbb}[1]{\ensuremath{\mathbb{#1}}}
\newcommand{\Ai}{\hbb{A}}
\newcommand{\Ci}{\hbb{C}}

\newcommand{\Ki}{\hbb{K}}
\newcommand{\Ni}{\hbb{N}}
\newcommand{\Qi}{\hbb{Q}}
\newcommand{\Ri}{\hbb{R}}

\newcommand{\Li}{\hbb{L}}

\newcommand{\hcal}[1]{\ensuremath{\mathcal{#1}}}

\newcommand{\Bc}{\hcal{B}}

\newcommand{\Nc}{\hcal{N}}
\newcommand{\Oc}{\hcal{O}}

\newcommand{\Ut}{\tilde U}

\theoremstyle{plain}
\newtheorem{thm}{Theorem}
\newtheorem{prop}{Proposition}
\newtheorem{cor}{Corollary}
\newtheorem{lem}{Lemma}
\theoremstyle{definition}
\newtheorem{dfn}{Definition}

\newtheorem{xmp}{Example}
\theoremstyle{remark}
\newtheorem{rem}{Remark}

\newcommand{\Edata}{\texttt{EdgeData}}

\newcommand{\AbhyankarMoh}{\texttt{AbhyankarTest}}
\newcommand{\irreducible}{\texttt{Irreducible}}
\newcommand{\Edgepoly}{\texttt{BoundaryPol}}
\newcommand{\NPA}{\texttt{ARNP}}

\newcommand{\PIrr}{\texttt{Quasi-Irreducible}}
\newcommand{\PDegenerated}{\texttt{Quasi-Degenerated}}
\newcommand{\AppRoot}{\texttt{AppRoot}}
\newcommand{\Expand}{\texttt{Expand}}
\newcommand{\Primitive}{\texttt{Primitive}}

\renewcommand{\O}{\textrm{\Oc}}
\newcommand{\Ot}{\O\tilde\,\,}
\newcommand{\M}{\textup{\textsf{M}}}
\newcommand{\dy}{{d}}
\newcommand{\dx}{{n}}

\newcommand{\val}[1][x]{\ensuremath{v_{#1}}}
\newcommand{\D}{\textrm{Data}}

\newcommand{\coef}{\textrm{coeff}}

\newcommand{\Res}{\textrm{Res}}

\newcommand{\Card}{\textrm{Card}}
\newcommand{\Conv}{\textrm{Conv}}

\newcommand{\vF}[1][]{{\ensuremath{\delta_{#1}}}}

\newcommand{\edgepoly}{boundary polynomial}
\newcommand{\True}{\texttt{True}}
\newcommand{\False}{\texttt{False}}
\newcommand{\Char}{\textrm{Char}}

\newcommand{\NP}{\Nc}
\newcommand{\algclos}[1]{\overline{#1}}
\newcommand{\tronc}[2]{{\left\lceil #1 \right\rceil}^{#2}}
\newcommand{\tc}[2][y]{\mbox{tc}_{#1}\left(#2\right)}

\begin{document}

\title{A quasi-linear irreducibility test in $\Ki[[x]][y]$}

\date{}
\author[1]{Adrien POTEAUX}
\author[2]{Martin WEIMANN}
 
\affil[1]{University of Lille, France}
\affil[2]{University of French Polynesia, France}

\maketitle

\begin{abstract}
  We provide an irreducibility test in the ring $\Ki[[x]][y]$ whose
  complexity is quasi-linear with respect to the discriminant
  valuation, assuming the input polynomial $F$ square-free and $\Ki$ a
  perfect field of characteristic zero or greater than $\deg(F)$. The
  algorithm uses the theory of approximate roots and may be seen as a
  generalisation of Abhyankhar's irreducibility criterion to the case
  of non algebraically closed residue fields.
\end{abstract}

\section{Introduction}
Factorisation of polynomials defined over a ring of formal power
series is an important issue of symbolic computation, with a view towards singularities of algebraic plane curves. In this paper, we develop a fast irreducibility test. In all
of the sequel, we assume that $F\in\Ki[[x]][y]$ is a square-free
Weierstrass polynomial defined over a perfect field $\Ki$ of
characteristic $0$ or greater than $\dy=\deg(F)$.  We let $\vF$ stand
for the $x$-valuation of the discriminant of $F$. We prove:
\begin{thm}\label{thm:main}
  We can test if $F$ is irreducible in $\Ki[[x]][y]$ with an expected
  $\Ot(\vF)$ operations over $\Ki$ and one univariate irreducibility
  test over $\Ki$ of degree at most $\dy$.
\end{thm}
If $F$ is irreducible, the algorithm computes also its discriminant
valuation $\vF$, its index of ramification $e$ and its residual degree
$f$. As usual, the notation $\Ot()$ hides logarithmic factors ; see
Section \ref{ssec:model} for details. Up to our knowledge, this
improves the best current complexity $\Ot(\dy\,\vF)$ \cite[Section
3]{PoWe17}.

Our algorithm is Las Vegas, due to the computation of primitive
elements\footnote{One should get a deterministic complexity
  $\O(\vF^{1+o(1)}\log^{1+o(1)}(\dy))$ thanks to the recent preprint
  \cite{HoLe18}.} in the residue field extensions. In particular, if
we test the irreducibility of $F$ in $\algclos{\Ki}[[x]][y]$, it
becomes deterministic without univariate irreducibility test.  The
algorithm extends to non Weierstrass polynomials, but with complexity
$\Ot(\vF+\dy)$ and at most two univariate irreducibility tests. If
$F\in \Ki[x,y]$ is given as a square-free bivariate polynomial of
bidegree $(\dx,\dy)$, we have $\vF< 2 \dx \dy$, hence our algorithm is
quasi-linear with respect to the arithmetic size $nd$ of the input
$F$. Moreover, we can avoid the square-free hypothesis in this
case. These extended results are discussed in Subsection
\ref{ssec:comments}.

\paragraph{Main ideas.} We recursively compute some well chosen
approximate roots $\psi_0,\ldots,\psi_g$ of $F$, starting with
$\psi_0$ the $d^{th}$ approximate roots of $F$. At step $k+1$, we
build the $(\psi_0,\ldots,\psi_k)$-adic expansion of $F$. We compute
an induced generalised Newton polygon of $F$ and check if it is
straigth. If not, then $F$ is reducible and the algorithm
stops. Otherwise, we construct a related \edgepoly{}
(quasi-homogeneous and defined over some field extenion of $\Ki$) and
test if it is the power of some irreducible polynomial. If not, then
$F$ is reducible and the algorithm stops. Otherwise, we deduce the
degree of the next approximate root $\psi_{k+1}$. The degrees of the
$\psi_k$'s are strictly increasing and $F$ is irreducible if and only
if we reach $\psi_g=F$. In order to perform a unique univariate
irreducibility test over $\Ki$, we rely on dynamic evaluation and
rather check if the \edgepoly{}s are powers of a square-free
polynomial.

\paragraph{Related results.} 
Factorisation of polynomials defined over a ring of formal power
series is an important issue in the algorithmic of algebraic curves,
both for local aspects (classification of plane curves singularities)
and for global aspects (integral basis of function fields \cite{vH94},
geometric genus of plane curves \cite{PoWe17}, bivariate factorisation
\cite{We16}, etc.) Probably the most classical approach for factoring
polynomials in $\Ki[[x]][y]$ is derived from the Newton-Puiseux
algorithm, as a combination of blow-ups (monomial transforms and
shifts) and Hensel liftings. This approach allows moreover to compute
the roots of $F$ - represented as fractional Puiseux series - up to an
arbitrary precision. The Newton-Puiseux algorithm has been studied by
many authors (see e.g. \cite{Du89, DeDiDu85, Wa00, Te90, Po08, PoRy11,
  PoRy12, PoRy15, PoWe17} and the references therein). Up to our
knowledge, the best current arithmetic complexity was obtained in
\cite{PoWe17}, using a divide and conquer strategy leading to a fast
Newton-Puiseux algorithm (hence an irreducibility test) which computes
the singular parts of all Puiseux series above $x=0$ in an expected
$\Ot(\dy\,\vF)$ operations over $\Ki$. There exists also other methods
for factorisation, as the Montes algorithm which allow to factor
polynomials over general local fields \cite{Mo99,GuMoNa12} with no
assumptions on the characteristic of the residue field. Similarly to
the algorithms we present in this paper, Montes et al. compute higher
order Newton polygons and \edgepoly{}s from the $\Phi$-adic expansion
of $F$, where $\Phi$ is a sequence of some well chosen polynomials
which is updated at each step of the algorithm. With our notations,
this leads to an irreducibility test in $\Ot(\dy^2+\vF^2)$
\cite[Corollary 5.10 p.163]{BaNaSt13} when $\Ki$ is a ``small enough''
finite field\footnote{This restriction on the field $\Ki$ is due to
  the univariate factorisation complexity. It could probably be
  avoided by using dynamic evaluation.}. In particular, their work
provide a complete description of \emph{augmented valuations},
apparently rediscovering the one of MacLane
\cite{Ma36a,Ma36b,Ru14}. The closest related result to this topic is
the work of Abhyanhar \cite{Ab89}, which provides a new irreducibility
test in $\Ci[[x]][y]$ based on approximate roots, generalised to
algebraically closed residue fields of arbitrary characteristic in
\cite{CoMo03}. No complexity estimates have been made up to our
knowledge, but we will prove that Abhyanhar's irreducibility criterion
is $\Ot(\vF)$ when $F$ is Weierstrass. In this paper, we extend this
result to non algebraically closed residue field $\Ki[[x]][y]$ of
characteristic zero or big enough. In some sense, our approach
establishes a bridge between the Newton-Puiseux algorithm, the Montes
algorithm and Abhyankar's irreducibility criterion.  Let us mention
also \cite{GaGw10,GaGw12} where an other irreducibility criterion in
$\algclos{\Ki}[[x]][y]$ is given in terms of the Newton polygon of the
discriminant curve of $F$, without complexity estimates.

\paragraph{Organisation.} In Section \ref{sec:ARNP}, we recall results
of \cite{PoRy15,PoWe17}, namely an improved version of the rational
Newton-Puiseux algorithm of Duval \cite{Du89}. From this algorithm we
fix several notations and define a collection $\Phi$ of minimal
polynomials of some truncated Puiseux series of $F$. We then show in
Section \ref{sec:phi} how to recover the edge data of $F$ from its
$\Phi$-adic expansion. In Section \ref{sec:psi}, we show that $\Phi$
can be replaced by a collection $\Psi$ of well chosen approximate
roots of $F$, which can be computed in the aimed complexity bound.
Section \ref{sec:comp} is dedicated to complexity issues and to the
proof of Theorem \ref{thm:main} ; in particular, we delay discussions
on truncations of powers of $x$ to this section. Finally, we give in
Section \ref{sec:absolute} a new proof of Abhyankhar's absolute
irreducibility criterion.


\section{A Newton-Puiseux type algorithm}\label{sec:ARNP}

\subsection{Classical definitions}\label{ssec:npalgo}
Let $F = \sum_{i=0}^\dy a_i(x)\,y^i=\sum_{i,j} a_{ij}x^j y^i \in
\Ki[[x]][y]$ be a Weierstrass polynomial, that is
$a_d=1$ and $a_i(0)=0$ for $i<d$ (the general case will be considered
in Section \ref{ssec:comments}). We let $\val$ stand for the usual
$x$-valuation of $\Ki[[x]]$.

\begin{dfn}\label{dnf:NP}
  The \textit{Newton polygon} of $F$ is the lower convex hull $\NP(F)$
  of the set of points $(i,\val(a_i))$ for $i=0,\ldots,\dy$.
\end{dfn}

It is well known that if $F$ is irreducible, then $\NP(F)$ is straight
(a single point being straight by convention). However, this condition
is not sufficient.

\begin{dfn}\label{dfn:edgePol}
  We call $\bar{F}:=\sum_{(i,j)\in \NP(F)} a_{ij} x^j y^i$ the
  \emph{\edgepoly{}} of $F$.
\end{dfn}
  
\begin{dfn}\label{dfn:degeneracy}
  We say that $F$ is \emph{degenerated} over $\Ki$ if its \edgepoly{}
  $\bar{F}$ is the power of an irreducible quasi-homogeneous
  polynomial.
\end{dfn}

In other words, $F$ is degenerated if and only if $\NP(F)$ is straight of
slope $-m/q$ with $q,m$ coprime, $q>0$, and if
\begin{equation}\label{eq:quasihom}
  \bar{F}=c \left(P\left( \frac{y^q}{x^{m}} \right) \,
    x^{m\deg(P)}\right)^N
\end{equation}
with $c\in \Ki^\times$, $N\in \Ni$ and $P\in \Ki[Z]$ monic and
irreducible.  We call $P$ the \emph{residual polynomial} of $F$. We
call the tuple $(q,m,P,N)$ the \emph{edge data} of the degenerated
polynomial $F$ and denote \Edata{} an algorithm computing this tuple.

\subsection{A Newton-Puiseux type irreducibility test}\label{ssec:Nksequence}
We can associate to $F$ a sequence of Weierstrass polynomials
$H_0,\ldots,H_g$ of strictly decreasing degrees $N_0,\ldots,N_g$ such
that either $N_g=1$ and $F$ is irreducible, either $H_g$ is not
degenerated and $F$ is reducible.

$\bullet $ \textbf{Rank $k=0$}. Let $N_0=\dy$ and $\Ki_0=\Ki$. We
define $c_0(x):=-\coef(F,y^{N_0-1})/N_0$ and let
\[
H_0(x,y):=F(x,y+c_0(x))\in \Ki_0[[x]][y].
\]
Then $H_0$ is a new Weierstrass polynomial of degree $N_0$ with no
terms of degree $N_0-1$. If $N_0=1$ or $H_0$ is not degenerated, we
let $g=0$.

$\bullet $ \textbf{Rank $k>0$}. Suppose given $\Ki_{k-1}$ a field
extension of $\Ki$ and $H_{k-1}\in \Ki_{k-1}[[x]][y]$ a degenerated
Weierstrass polynomial of degree $N_{k-1}$, with no terms of degree
$N_{k-1}-1$. Denote $(q_k,m_k,P_k,N_k)$ its edge data and
$\ell_k=\deg(P_k)$. We let $z_k$ stands for the residue class of $Z_k$
in the field $\Ki_k:=\Ki_{k-1}[Z_k]/(P_k(Z_k))$. We define $(s_k,t_k)$
to be the unique positive integers such that $q_ks_k-t_k m_k=1$,
$0\le t_k < q_k$. As $H_{k-1}$ is degenerated, we deduce from
\eqref{eq:quasihom} that
\begin{equation}\label{eq:Gk}
  H_{k-1}(z_k^{t_k} x^{q_k},x^{m_k}(y+z_k^{s_k}))=x^{q_km_k\ell_k N_k} \, G_k\, V_k,
\end{equation}
where $V_k\in \Ki_k[[x,y]]$ is a unit and $G_k\in \Ki_k[[x]][y]$ is a
Weierstrass polynomial of degree $N_k$ which can be computed up to an
arbitrary precision via Hensel lifting. We let
$c_k:=-\coef(G_k,y^{N_k-1})/N_{k}$ and define
\begin{equation}\label{eq:Hk}
  H_k(x,y)=G_k(x,y+c_k(x))\in \Ki_k[[x]][y].
\end{equation}
It is a degree $N_k$ Weierstrass polynomial with no terms of degree
$N_k-1$.

$\bullet$ \textbf{The $N_k$-sequence stops.} We have the relations
$N_k=q_k \ell_k N_{k-1}$. As $H_{k-1}$ is degenerated with no terms of
degree $N_{k-1}-1$, we must have $q_k \ell_k>1$. Hence the sequence of
integers $N_0,\ldots,N_k$ is strictly decreasing and there exists a
smallest index $g$ such that either $N_g=1$ and $H_g=y$ or $N_g>1$ and
$H_g$ is not degenerated. We collect the edge data of the polynomials
$H_0,\ldots,H_{g-1}$ in a list
\[
  \D(F):=\big((q_1,m_1,P_1,N_1),\ldots,(q_{g},m_{g},P_{g},N_{g})\big).
\]
Note that $m_{k}>0$ for all $1\leq k \le g$.

\begin{prop}\label{prop:NPA}
  The polynomial $F$ is irreducible if and only if $N_g=1$.
\end{prop}

\begin{proof}
  Follows from the rational Puiseux algorithm of Duval \cite{Du89}
  (which is based on the transform \eqref{eq:Gk}) combined with the
  ``Abhyankhar's trick'' \eqref{eq:Hk} introduced in
  \cite{PoRy15}.
\end{proof}

Following \cite{PoWe17}, we denote by \NPA{} the underlying
algorithm. By considering suitable sharp truncation bounds, it is
shown in \cite[Section 3]{PoWe17} that this algorithm performs an
expected $\Ot(\dy\,\vF)$ arithmetic operations (this requires
algorithmic tricks, especially dynamic evaluation and primitive
representation of residue fields). Unfortunately, the worst case
complexity of this algorithm is $\Omega(\dy\,\vF)$, which is too high
for our purpose. The main reason is that computing the intermediate
polynomials $G_k$ in \eqref{eq:Gk} via Hensel lifting up to sufficient
precision requires to compute
$H_{k-1}(z_k^{t_k} x^{q_k},x^{m_k}(y+z_k^{s_k}))$, that might have a
size $\Omega(\dy\,\vF)$, as shows the following example.
\begin{xmp}\label{xmp:sharp}
  Consider $F=(y^\alpha-x^2)^2+x^\alpha\in\Ci[[x]][y]$ with $\alpha>4$
  odd. We have $\dy=2\,\alpha$, $\vF=2\,\alpha^2-4\,\alpha+4$, $H_0=F$
  and $q_1, m_1, z_1$ are respectively $\alpha, 2$ and $1$. Applying
  results of \cite[Section 3]{PoWe17}, one can show that an optimal
  truncation bound to compute $G_1$ is $\alpha^2-4\,\alpha+1$. But the
  size of
  $H_0(x^\alpha, x^2\,(y+1))/x^{4\,\alpha} \mod
  x^{\alpha^2-4\,\alpha+1}$ is $\Theta(\alpha^3)=\Theta(\dy\,\vF)$.
\end{xmp}
To solve this problem we will rather compute the \edgepoly{}
$\bar{H}_k$ using the $(\psi_0,\ldots,\psi_k)$-adic expansions of $F$,
where the $\psi_k$'s are well chosen approximate roots. As a first
step towards the proof of this result, we begin by using a sequence
$(\phi_0,\ldots,\phi_k)$ of minimal polynomials of $F$ that we now
define.

\subsection{Minimal polynomials of truncated rational Puiseux
  expansions}

\paragraph{Rational Puiseux Expansions.} We keep notations of Section
\ref{ssec:Nksequence}. We denote $\pi_0(x,y)=(x,y+c_0(x))$ and define
inductively $\pi_k=\pi_{k-1}\circ \sigma_k$ where
\begin{equation}\label{eq:tausigma}
  \sigma_k(x,y):=(z_k^{t_k} x^{q_k},x^{m_k}(y+z_k^{s_k}+c_k(x)))
\end{equation}
for $k\ge 1$. It follows from \eqref{eq:Gk} and \eqref{eq:Hk} that
there exists $v_k(F)\in \Ni$ such that
\begin{equation}\label{eq:pikHk}
  \pi_k^* F = x^{v_k(F)}\, H_k\, U_k \in \Ki_k[[x]][y],
\end{equation}
with $U_k (0, 0) \in \Ki_k^\times$. This key point will be used
several time in the sequel.

We deduce from \eqref{eq:tausigma} that
\begin{equation}\label{eq:pikxy}
  \pi_k (x, y ) = (\mu_k x^{e_k},\alpha_k x^{r_k} y + S_{k} (x)),
\end{equation}
where $e_k:=q_1\cdots q_k$ (the ramification index discovered so far),
$\mu_k,\alpha_k \in \Ki_k^{\times}$, $r_k\in \Ni$ and
$S_k \in \Ki_k[[x]]$ satisfies $\val(S_k)\le r_k$.  Following
\cite{PoWe17}, we call the pair
\[
  \pi_k(x,0)=(\mu_k x^{e_k}, S_k(x))
\]
a (truncated) rational Puiseux parametrisation. This provides the
roots of $F$ (namely Puiseux series) truncated up to precision
$\frac{r_k}{e_k}$, that increases with $k$ \cite[Section 3.2]{PoWe17}.

\paragraph{Minimal polynomials.} It can be shown that the exponent
$e_k$ is coprime with the gcd of the support of $S_k$, and that the
coefficients of the parametrisation $(\mu_k x^{e_k}, S_k)$ generate
the current residue field extension $\Ki_k$ over $\Ki$ (see
e.g. \cite[Theorems 3 and 4]{Du89}).  It follows that there exists a unic
monic irreducible polynomial $\phi_k\in \Ki[[x]][y]$ such that
\begin{equation}\label{eq:phik}
  \phi_k (\mu_k x^{e_k}, S_{k}(x)) = 0 \textrm{ and } d_k:=\deg(\phi_k)=e_k f_k,
\end{equation}
where $f_k:=[\Ki_k:\Ki]=\ell_1\cdots \ell_k$. We call $\phi_k$ the
\emph{$k^{th}$ minimal polynomial} of $F$. Note that $\phi_0=y-c_0(x)$
and that we have the relations $\dy=N_k\,d_k$ for $k=0,\ldots,g$.

By construction, a function call $\NPA(\phi_k)$ generates the same
transformations $\pi_i$ for $i\le k$. In particular, we have
\[
  \D(\phi_k)=\left((q_1,m_1,P_1,N'_1),\ldots,(q_k,m_k,P_k,N_k'=1)\right)
  \text{ with } N'_i:=N_i/N_k.
\]


\section{Edge data from the $\Phi$-adic expansion}
\label{sec:phi}

Let us fix an integer $0\le k \le g$ and assume that $N_k>1$. We keep
using notations of Section \ref{sec:ARNP}. Assuming that we know the
edge data $(q_1,m_1,P_1,N_1),\ldots,(q_k,m_k,P_k,N_k)$ of the
Weierstrass polynomials $H_0,\ldots, H_{k-1}$, together with the
minimal polynomials $\phi_0,\dots,\phi_k$, we want to compute the
\edgepoly{} of the next Weierstrass polynomial $H_k$.  In the
following, we will omit for readibility the index $k$ for the sets
$\Phi$, $\Bc$, $V$ and $\Lambda$ defined below.

\subsection{Main results}\label{ssec:phi-main}

\paragraph{$\Phi$-adic expansion.} We denote $\phi_{-1}:=x$ and let
$\Phi=(\phi_{-1},\phi_0,\ldots,\phi_{k})$. Let
\begin{equation}\label{eq:Bc}
  \Bc := \{(b_{-1} ,\ldots,b_k)\in \Ni^{k+2} \,\,, \,
  \, b_{i-1}<q_i\,\ell_i\,, i=1,\ldots,k\}
\end{equation}
and denote $\Phi^B:=\prod_{i=-1}^k\phi_i^{b_i}$. Thanks to the
relations $\deg(\phi_i) = \deg(\phi_{i-1})q_i \ell_i$ for all
$1 \le i \le k$, an induction argument shows that $F$ admits a unique
expansion
\[
  F = \sum_{B\in \Bc} f_B \Phi^B, \quad f_B\in \Ki.
\]
We call it the $\Phi$-adic expansion of $F$. We have $b_k\leq N_k$
while we do not impose any \textit{a priori} condition to the powers
of $\phi_{-1}=x$ in this expansion. The aim of this section is to show
that one can compute $\bar{H}_k$ from the $\Phi$-adic expansion of
$F$.

\paragraph{Newton polygon.}
Consider the semi-group homomorphism
\[
  \begin{array}{rcl}
    v_k:\ (\Ki[[x]][y],\times)&\to& (\Ni\cup\{\infty\},+)\\%
    H & \mapsto & v_k (H) := \val (\pi_k^* H),
  \end{array}
\]
From \eqref{eq:pikxy}, we deduce that the pull-back morphism $\pi_k^*$
is injective, so that $v_k$ defines a discrete valuation. This is a
valuation of transcendence degree one, thus an augmented valuation
\cite[Section 4.2]{Ru14}, in the flavour of MacLane valuations
\cite{Ma36a,Ma36b,Ru14} or Montes valuations
\cite{Mo99,GuMoNa12}. Note that $v_0(H)=\val(H)$. We associate to $\Phi$ the vector
\[
  V := (v_k(\phi_{-1}), \ldots, v_k(\phi_k)),
\]
so that $v_k(\Phi^B)=\langle B , V\rangle$, where
$\langle \,,\,\rangle$ stands for the usual scalar product. For all
$i\in \Ni$, we define the integer
\begin{equation}\label{eq:wj}
  w_i := \min \left\{\langle B, V\rangle, \,\, b_k=i,\,\,f_B \ne 0\right\}-v_k(F)
\end{equation}
with convention $w_i := \infty$ if the minimum is taken over the empty
set. 

\begin{thm}
  \label{thm:NPphi}
  The Newton polygon of $H_k$ is the lower convex hull of
  $(i,w_i)_{0\leq i\leq N_k}$.
\end{thm}

This result leads us to introduce the sets
\[
  \Bc(i):=\{B\in\Bc ; b_k=i\} \textrm{ and } \Bc(i,w) := \{B \in
  \Bc(i) \ | \ \, \langle B, V \rangle = w\}
\]
for all $i\in\Ni$ and all
$w\in \Ni\cup\{\infty\}$, with convention $\Bc(i,\infty)=\emptyset$.

\paragraph{Boundary polynomial.} Consider the semi-group homomorphism
\[
  \begin{array}{rcl}
    \lambda_k:\ (\Ki[[x]][y ],\times)&\to& (\Ki_k,\times)\\%
    H & \mapsto & \lambda_k (H) := \tc{\left(\frac{\pi_k^{*}
                  (H)}{x^{v_k(H)}}\right)_{|x=0}}
  \end{array}
\]
with convention $\lambda_k(0)=0$, and where $\textrm{tc}_y$ stands for
the trailing coefficient with respect to $y$. We associate to $\Phi$
the vector
\[
  \Lambda := (\lambda_k (\phi_{-1}),\ldots,\lambda_k(\phi_k))
\]
and denote
$\Lambda^B:=\prod_{i=-1}^k \lambda_k(\phi_i)^{b_i}=\lambda_k(\Phi^B)$.
Note that $\Lambda^B\in \Ki_k$ is non zero for all $B$. We obtain the
following result:

\begin{thm}\label{thm:EdgePoly} Let $B_0 := (0,\ldots,0, N_k)$. The
  \edgepoly{} $\bar H_k$ of $H_k$ equals
  \begin{equation}\label{eq:barHk}
    \bar{H}_k =\sum_{(i,w_i)\in \NP(H_k)} \left(\sum_{B\in \Bc(i,w_i+v_k(F))}f_B \Lambda^{B-B_0} \right)x^{w_i} y^i.
  \end{equation}
\end{thm}

Combined with the formulas \eqref{eq:update} of Section
\ref{ssec:resume} for the vectors $V$ and $\Lambda$, Theorems
\ref{thm:NPphi} and \ref{thm:EdgePoly} give an efficient way to decide
if the Weierstrass polynomial $H_k$ is degenerated, and if so, to
compute its edge data.

\begin{xmp}\label{xmp:k0}
  If $k=0$, we have by definition $V=(1,0)$ and $\Lambda=(1,1)$ while
  $v_0(F)=\val(H_0)=0$. Assuming $H_0=\sum_{j=0}^\dy a_i (x) y^i$, we
  find $w_i=\val(a_i)$ and Theorem \ref{thm:NPphi} stands from
  Definition \ref{dnf:NP}. Moreover, $\Bc(i,w_i)$ is then reduced to
  the point $(i,w_i)$ and Theorem \ref{thm:EdgePoly} stands from
  Definition \ref{dfn:edgePol}.
\end{xmp}

\subsection{Proof of Theorems \ref{thm:NPphi} and \ref{thm:EdgePoly}}

Let us first establish some basic properties of the minimal
polynomials $\phi_i$ of $F$. Given a ring $\Ai$, we denote by
$\Ai^\times$ the subgroup of units.  Note that $U\in\Ai[[x,y]]^\times$
if and only if $U(0,0)\in \Ai^\times$.  For $-1 \leq i \leq k$, we
introduce the notations
\[
  v_{k,i} := v_k (\phi_i )= \val(\pi_k^*(\phi_i)) \textrm{ and }
   \lambda_{k,i} := \lambda_k (\phi_i
  )=\tc{\left(\frac{\pi_k^*(\phi_i)}{x^{v_{k,i}}}\right)_{|x=0}}.
\]
\begin{lem}\label{lem:PikPhi} Let $-1\le i \le k$. There exists
  $U_{k,i}\in \Ki_k [[x,y]]^\times$ with
  $U_{k,i} (0, 0) = \lambda_{k,i}$ s.t.:
  \begin{enumerate}
  \item\label{enum:uki} $\pi_k^* (\phi_i ) = x^{v_{k,i}}\,U_{k,i}$ if $i < k$,
  \item\label{enum:ukk} $\pi_k^* (\phi_k) = x^{v_{k,k}}\, y\,U_{k,k}$.
  \end{enumerate}
\end{lem}

\begin{proof} As \NPA{}($\phi_k$) generates the same transform
  $\pi_k$, we deduce from \eqref{eq:pikHk}:
  \[
    \pi_k^*(\phi_k) = x^{v_k(\phi_k)}\,(y+\beta(x))\,U(x,y)
  \]
  with $U \in \Ki_k [[x,y]]^\times$ and $\beta\in \Ki_k [[x]]$. From
  \eqref{eq:pikxy} and \eqref{eq:phik}, we get
  $x^{v_{k,k}} U(x,0)\,\beta(x) = \phi_k(\mu_k x^{e_k},S_{k}) = 0$,
  i.e. $\beta = 0$. Second equality follows, since
  $U(0, 0) = \lambda_{k,k}$ by definition of $\lambda_k$. First
  equality is then obtained by applying the pull-backs
  $\sigma_j^*$, $j=i+1,\ldots,k$ to
  $\pi_i^*(\phi_i)=x^{v_{i,i}}\,y\,U_{i,i}$.
\end{proof}

\begin{cor}\label{cor:res}
  With the standard notations for intersection multiplicities and
  resultants, we have for any $G\in\Ki[[x]][y]$ Weierstrass:
  \[
    v_k (G) =\frac{(G,\phi_k)_0}{f_k}= \frac{\val
      (\Res_y(G,\phi_k))}{f_k}.
  \]
\end{cor}

\begin{proof}
  As $\val(S_k)\leq r_k$, we get from \eqref{eq:pikxy}
  $v_k (G) = \val(\pi_k^*(G)) = \val(G(\mu_k
  x^{e_k},S_k(x)))$.
  But this last integer coincides with the intersection multiplicity
  of $\phi_i$ with any one of the $f_k$ conjugate plane branches
  (i.e. irreducible factor in $\algclos{\Ki}[[x]][y]$) of
  $\phi_k$. The first equality follows. The second is well known (the
  intersection multiplicity at $(0, 0)$ of two Weierstrass polynomials
  coincides with the $x$-valuation of their resultant).
\end{proof}

\begin{lem}\label{lem:vLambda} We have  initial conditions
  $v_{0,-1}=1$, $v_{0,0}=0$, $\lambda_{0,-1}=1$ and $\lambda_{0,0}=1$.
  Let $k\ge 1$. The following relations hold $($we recall
  $q_k s_k - m_k t_k = 1$ with $0\le t_k < q_k):$
  \begin{enumerate}
  \item $v_{k,k-1} = q_k v_{k-1,k-1} + m_k$
  \item $v_{k,i} = q_k v_{k-1,i}$ for all $-1\le i < k- 1$.
  \item
    $\lambda_{k,k-1} = \lambda_{k-1,k-1} z_k^{t_k v_{k-1,k-1}
      +s_{k}}$.
  \item $\lambda_{k,i} =\lambda_{k-1,i}z_k^{t_k v_{k-1,i}}$ for all
    $-1 \le i < k - 1$.
  \end{enumerate}
\end{lem}

\begin{proof} Initial conditions follow straightforwardly from the
  definitions. Using point \ref{enum:ukk} of Lemma \ref{lem:PikPhi} at
  rank $k-1$ and equality
  $\pi_k^*(\phi_{k-1})=\sigma_k^*\circ \pi_{k-1}^*(\phi_{k-1})$, we
  get
  \[
    \pi_k^*(\phi_{k-1})=z_k^{t_k v_{k-1,k-1}} x^{q_k v_{k-1,k-1} +
      m_k}(y+z_k^{s_k}+c_k) U_{k-1,k-1}(z_k^{t_k}
    x^{q_k},x^{m_k}(y+z_k^{s_k}+c_k)).
  \]
  As $c_k(0)= 0$, $m_k > 0$ and $z_k\ne 0$, it follows that
  \[
    \pi_k^*(\phi_{k-1})=z_k^{t_k v_{k-1,k-1}+s_k}
    x^{q_k\,v_{k-1,k-1}+m_k} \tilde{U}(x,y)
  \]
  with $\tilde{U}(0,0)=U_{k-1,k-1}(0,0)$, that is $\lambda_{k-1,k-1}$
  by point \ref{enum:ukk} of Lemma \ref{lem:PikPhi}. Items $1$ and $3$
  follow. Similarly, using point \ref{enum:uki} of Lemma
  \ref{lem:PikPhi} at rank $k-1$, we get for $i<k-1$
  \[
    \pi_k^*(\phi_i)=\sigma_k^*\circ \pi_{k-1}^*(\phi_i)=z_k^{t_k
      v_{k-1,i}} x^{q_k v_{k-1,i}} U_{k-1,i}(z_k^{t_k}
    x^{q_k},x^{m_k}(y+z_k^{s_k}+c_k(x))).
  \]
  As $U_{k-1,i}(0,0)=\lambda_{k-1,i}\ne 0$ once again by point
  \ref{enum:uki} of Lemma \ref{lem:PikPhi}, items $2$ and $4$ follow.
\end{proof}

The proof of both theorems is based on the following key result:

\begin{prop}\label{prop:key} For all $i, w \in\Ni$, the family
  $\left(\Lambda^B , B \in \Bc(i,w)\right)$ is free over $\Ki$. In
  particular, $\Card \,\, \Bc(i,w)\le f_k$.
\end{prop}

\begin{proof} We show this property by induction on $k$. If $k = 0$,
  the result is obvious since $\Bc(i,w)=\{(i,w)\}$ and
  $\Lambda=(1,1)$. Suppose $k > 0$. As $\lambda_{k,k}$ is invertible
  and $b_k=i$ is fixed, we are reduced to show that the family
  $\left(\Lambda^B , B \in \Bc(0,w)\right)$ is free for all
  $w\in \Ni$. Suppose given a $\Ki$-linear relation
  \begin{equation}\label{eq:cBlambdaB}
    \sum_{B\in \Bc(0,w)} c_B \Lambda^B =   \sum_{B\in \Bc(0,w)} c_B \lambda_{k,-1}^{b_{-1}}\cdots\lambda_{k,k-1}^{b_{k-1}}= 0.
  \end{equation}
  Using $b_k = 0$, points 3 and 4 in Lemma \ref{lem:vLambda} give
  $\Lambda^B = \mu_B z_k^{N_B}$ where
  \[
    \mu_B = \prod_{j=-1}^{k-1} \lambda_{k-1,j}^{b_j} \in
    \Ki_{k-1} \textrm{ and } N_B = b_{k-1} s_k +
    t_k\sum_{j=-1}^{k-1} b_j v_{k-1,j}.
  \]
  Points 1 ($q_k\,v_{k-1,k-1}=v_{k,k-1}-m_k$) and 2
  ($q_k\,v_{k-1,j} = v_{k,j}$) in Lemma \ref{lem:vLambda} give
  \begin{equation}\label{eq:qkNB}
    q_k N_B = b_{k-1} (q_k s_k - m_k t_k ) + t_k \sum_{j=-1}^{k-1} b_j
    v_{k,j} = b_{k-1} + t_k w,
  \end{equation}
  the second equality using $\langle B, V \rangle = w$ and
  $b_k=0$. Since $0\le b_{k-1} < q_k \ell_k$ and $N_B$ is an integer,
  it follows from \eqref{eq:qkNB} that $N_B=n+\alpha$ where
  $n=\lceil t_k w/q_k\rceil$ and $0\leq \alpha <\ell_k$. Dividing
  \eqref{eq:cBlambdaB} by $z_k^n$ and using
  $\Lambda^B = \mu_B z_k^{N_B}$, we get
  \[
    \sum_{\alpha=0}^{\ell_k -1} a_{\alpha} z_{k}^\alpha = 0,
    \textrm{ where } a_{\alpha} =\sum_{B\in \Bc(0,w) ,N_B =
      \alpha+n} c_{B} \mu_B .
  \]
  Since $a_{\alpha} \in \Ki_{k-1}$ and $z_k\in \Ki_k$ has minimal
  polynomial $P_k$ of degree $\ell_k$ over $\Ki_{k-1}$, this implies
  $a_{\alpha} = 0$ for all $0\leq \alpha < \ell_k$, i.e., using
  \eqref{eq:qkNB}:
  \[
    \sum_{\substack{B\in \Bc(0,w)\\ b_{k-1}=q_k\,(\alpha+n)-t_k\,w}}
    c_{B} \lambda_{k-1,-1}^{b_{-1}}\cdots \lambda_{k-1,k-1}^{b_{k-1}}
    = 0.
  \]
  By induction, we get $c_B=0$ for all $B\in \Bc(0,w)$, as required.
  The first claim is proved. The second claim follows immediately
  since $\Lambda^B \in \Ki_k$ is non zero for all $B$.
\end{proof}

\begin{cor}\label{cor:main}
  Consider $G=\sum_{B\in\Bc(i)} g_B \Phi^B$ non zero.  Then
  $\pi_k^*(G)=x^{w}\,y^i\,\Ut $ with $\Ut\in \Ki_k[[x,y]]^ \times$,
  $w=\min_{ g_B\ne 0} \langle B,V\rangle$ and
  $\Ut(0,0)=\sum_{B\in \Bc(i,w)}g_B\Lambda^B\ne 0$.  In particular,
  $v_k(G)=w$ and $\lambda_k(G)=\Ut(0,0)$.
\end{cor}

\begin{proof}
  By linearity of $\pi_k^*$, denoting
  $U = (U_{k,-1} ,\ldots , U_{k,k} )$ with $U_{k,i}$ defined in Lemma
  \ref{lem:PikPhi}, we have
  \[
    \pi_k^*(G) = \left(\sum_{B\in\Bc(i)}g_B \, x^{<B,V>}\,U^B \right)\,y^i
    \text{ with } U(0,0)=\Lambda.
  \]
  Letting $w=\min_{ g_B\ne 0} \langle B,V\rangle$, we deduce
  \[
    \pi^{*}_k(G) =\left(\sum_{B\in \Bc(i,w)} g_B \Lambda^B + R \right)
    x^{w} y^i \textrm{ where } R\in \Ki_k[[x,y]] \textrm{ satisfies
    } R(0,0)=0.
  \]
  As $\sum_{B\in\Bc(i,w)}g_B \Lambda^B \neq 0$ by Proposition
  \ref{prop:key}, the first two equalities follows. The last two
  equalities follow from the definitions of $v_k(G)$ and
  $\lambda_k(G)$.
\end{proof}

\paragraph{Proof of Theorems \ref{thm:NPphi} and \ref{thm:EdgePoly}.}
We prove both theorems simultaneously. We may write
$F=\sum_{i=0}^{N_k} \sum_{B\in \Bc(i)} f_B\Phi^B$. Hence, Corollary
\ref{cor:main} combined with the definition of $w_i$ and the linearity
of $\pi_k^*$ implies
\[
  F_k :=\frac{\pi^{*}_k(F)}{x^{v_k(F)}}
  =\sum_{i=0}^{N_k} \,x^{w_i}\, y^i\, \Ut_i
\]
where $\Ut_i\in \Ki_k[[x,y]]$ is $0$ if $w_i=\infty$, and
$\Ut_i(0,0)=\sum_{B\in \Bc(i,w_i+v_k(F))} f_B \Lambda^B\ne 0$
otherwise.  As $H_k$ is Weierstrass of degree $N_k$, it follows from
this formula combined with \eqref{eq:pikHk}, that $\NP(H_k)$
coincides with the lower convex hull of the points $(i,w_i)$,
$i=0,\ldots,N_k$, proving Theorem \ref{thm:NPphi}. More precisely, we
deduce that there exists $\mu \in \Ki_k^\times$ such that
\[
  \mu\bar{H}_k=\sum_{(i,w_i )\in\NP(H_k)} \left(\sum_{B\in \Bc(i,w_i+v_k(F))} f_B \Lambda^B \right) x^{w_i} y^i.
\]
As $\bar{H}_k$ is Weierstrass of degree $N_k$, then $w_{N_k} = 0$
and $w_i >0$ for $i< N_k$. The previous equation forces
\[
  \mu=\sum_{B\in \Bc(N_k,v_k(F))} f_B \Lambda^B .
\]
But $F$ and $\phi_k$ being monic of respective degrees $\dy$ and
$d_k$, the vector $B_0 = (0,\ldots 0, N_k) \in \Bc$ is the unique
exponent in the $\Phi$-adic expansion of $F$ with last coordinate
$b_k = N_k=d/d_k$ and we have moreover $f_{B_0} = 1$.  This forces
$\Bc(N_k,v_k(F)) = \{B_0\}$ and we get $\mu = \Lambda^{B_0}$, thus
proving Theorem \ref{thm:EdgePoly}. $\hfill\square$

\subsection{Formulas for $\lambda_k(\phi_k)$ and $v_k(\phi_k)$}
In order to use Theorems \ref{thm:NPphi} and \ref{thm:EdgePoly} for
computing the edge data of $H_k$, we need to compute
$v_{k,k}=v_k(\phi_k)$, $\lambda_{k,k}=\lambda_k(\phi_k)$, $v_k(F)$ and
$\lambda_k(F)$ in terms of the previously computed edge data
$(q_1,m_1,P_1,N_1),\ldots,(q_k,m_k,P_k,N_k)$ of $F$. We begin with the
following lemma:

\begin{lem}\label{lem:FPhi}
  Let $0\le k \le g$. We have $v_k(F)= N_k v_{k,k}$ and
  $\lambda_k(F) =\lambda_{k,k}^{N_k}$.
\end{lem}

\begin{proof}
  We have shown during the proof of Theorems \ref{thm:NPphi} and
  \ref{thm:EdgePoly} that $\Bc(N_k,v_k(F))=\{B_0\}$ with
  $B_0=(0,\ldots,0,N_k)$. By definition of $\Bc(N_k,v_k(F))$, we get
  the first point. From the definition of $\lambda_k$, we have
  $\lambda_k(F)=\tc{F_k(0,y)}=\tc{\bar{F_k}(0,y)}$ and we have shown
  that $\bar{F_k}(0,y)=\Lambda^{B_0}\bar{H_k}(0,y)$. Since $\bar{H_k}$
  is monic, we deduce $\tc{\bar{F_k}(0,y)}=\Lambda^{B_0}$.
\end{proof}

\begin{prop}\label{prop:vkk}
  For any $1\le k \le g$, we have the equalities
  \[
    v_{k,k}= q_k \ell_k v_{k,k-1}\textrm{ and } \lambda_{k,k}=q_k
    z_k^{1-s_k-\ell_k}P_k'(z_k)\lambda_{k,k-1}^{q_k \ell_k}.
  \]
\end{prop}

\begin{proof}
  To simplify the notations of this proof, let us denote
  $w=v_{k-1}(\phi_k)$, $\gamma=\lambda_{k-1}(\phi_k)$ and
  $(m,q,s,t,\ell,z)=(m_k,q_k,s_k,t_k,\ell_k,z_k)$.  Remember from
  Section \ref{sec:ARNP} that by definition of $\phi_k$, both $\phi_k$
  and $F$ generate the same transformations $\sigma_i$ and $\tau_i$
  for $i\leq k$. As in \eqref{eq:pikHk}, there exists
  $\Ut_{k-1}\in\Ki[[x,y]]^\times$ satisfying $\Ut_{k-1}(0,0)=\gamma$
  and $\tilde{H}_{k-1}\in\Ki[[x]][y]$ Weierstrass of degree $q\,\ell$
  such that $\pi_{k-1}^{*}(\phi_k)= x^w \tilde{H}_{k-1} \Ut_{k-1}$,
  where
  \[
  \tilde H_{k-1}(x,y)=P_k(x^{-m}\,y^q)\,x^{m \ell}+\sum_{m i + q j > m q \ell} h_{ij}\, x^j\, y^i.
  \]
  We deduce that there exists $R_0,R_1,R_2\in \Ki_k[[x,y]]$ such that
  \begin{eqnarray*}
    \pi_{k}^{*} (\phi_k) & := & (\pi_{k-1}^{*} (\phi_k)) (z^t x^q,x^m(y+z^s+c_k(x))) \\
                         & =  &z^{tw}x^{qw}\left(z^{tm\ell}\,x^{mq\ell}(G_k+x R_0)\right)\,\left(\gamma+x\,R_1+y\,R_2\right)%
  \end{eqnarray*}%
  where we let
  $G_k(x,y):=P_k(z^{-tm}(y+z^s+c_k(x))^q)\in\Ki_k[[x]][y]$. It follows
  that there exists $R\in \Ki_k[[x,y]]$ such that
  \begin{equation}\label{eq:pikstar}
    \pi_{k}^{*} (\phi_k) = z^{t(w+m\ell)}\,x^{q\,(w+m\ell)} \left((\gamma+y\, R_2)\,G_k+x\,R\right).
  \end{equation}      
  As $G_k(0,y)$ is not identically zero, we deduce from
  \eqref{eq:pikstar} that $v_k(\phi_k)=q\,(w+m\ell)$. Using Lemma
  \ref{lem:FPhi} for $F=\phi_k$ and the valuation $v_{k-1}$, together
  with Point 1 of Lemma \ref{lem:vLambda}, we have
  $w+m\ell=\ell v_{k,k-1}$, which implies $v_{k,k}=q\,\ell\,v_{k,k-1}$
  as expected.  Using $c_k(0)=0$ and the relation $sq-tm=1$, we get
  $G_k(0,0)=P_k(z_k)=0$ and
  $\partial_y G_k (0,0)=q z^{1-s} P'_k(z)\neq 0$. Combined with
  \eqref{eq:pikstar}, this gives
  \[
    \lambda_{k,k}=\gamma z^{t\,\ell\,v_{k,k-1}} \left(q z^{1-s}
      P'_k(z)\right) =
    \gamma\,z^{q\,\ell\,t\,v_{k-1,k-1}+\ell\,t\,m+1-s}\,q\,P_k'(z),
  \]
  the second equality using Point 1 of Lemma \ref{lem:vLambda} once
  again. Now, using Lemma \ref{lem:FPhi} for $F=\phi_k$ and the
  morphism $\lambda_{k-1}$, we get $\gamma=\lambda_{k-1,k-1}^{\ell q}$
  i.e.
  $\lambda_{k,k} =q P'_k(z)\lambda_{k,k-1}^{q \ell} z^{1-s-\ell}$.
\end{proof}

\subsection{Simple formulas for $V$ and $\Lambda$}\label{ssec:resume}
For convenience to the reader, let us summarize the formulas which
allow to compute in a simple recursive way both lists
$V=(v_{k,-1},\ldots, v_{k,k})$ and
$\Lambda=(\lambda_{k,-1},\ldots,\lambda_{k,k})$.

If $k=0$, we let $V=(1,0)$ and $\Lambda=(1,1)$. Assume $k\ge 1$. Given
the lists $V$ and $\Lambda$ at rank $k-1$ and given the $k$-th edge
data $(q_k,m_k,P_k,N_k)$, we update both lists at rank $k$ thanks to
the formulas:
\begin{equation}\label{eq:update}
  \begin{cases}
    v_{k,i} = q_k v_{k-1,i} \quad {\small\text{$-1\le i<k-1$}}\\%
    v_{k,k-1} = q_k v_{k-1,k-1}+m_k\\%
    v_{k,k} = q_k\ell_k v_{k,k-1}
  \end{cases}
  \ 
  \begin{cases}
    \lambda_{k,i} = \lambda_{k-1,i}z_k^{t_k v_{k-1,i}}\quad
    {\small\text{$-1\le i<k-1$}}\\%
    \lambda_{k,k-1} = \lambda_{k-1,k-1} z_k^{t_k v_{k-1,k-1}+s_{k}}\\%
    \lambda_{k,k} = q_k
    z_k^{1-s_k-\ell_k}P_k'(z_k)\lambda_{k,k-1}^{q_k
      \ell_k}
  \end{cases}
\end{equation}
where $q_ks_k-m_k t_k=1$, $0\le t_k< q_k$ and $z_k=Z_k\mod P_k$.


\section{From minimal polynomials to approximate roots}\label{sec:psi}

Given $\Phi=(\phi_{-1},\ldots,\phi_k)$ and $F=\sum f_B \Phi^B$ the
$\Phi$-adic expansion of $F$, the updated lists $V$ and $\Lambda$
allow to compute in an efficient way the \edgepoly{} $\bar H_k$ using
 formulas \eqref{eq:wj} and \eqref{eq:barHk}. Unfortunately, we
do not know a way to compute the minimal polynomials $\phi_k$ in our
aimed complexity bound: the computation of the
$y^{N_k-1}$ coefficient of $G_k$ up to some suitable
precision might cost $\Omega(\dy\,\vF)$ as explained in
Section \ref{sec:ARNP}.

We now show that the main conclusions of all previous results remain
true if we replace $\phi_k$ by the $N_k^{th}$-approximate root
$\psi_k$ of $F$, with the great advantage that these approximate roots
can be computed in the aimed complexity (see Section
\ref{sec:comp}). Up to our knowledge, such a strategy was introduced
by Abhyankar who developped in \cite{Ab89} an irreducibility criterion
in $\algclos{\Ki}[[x,y]]$ avoiding any Newton-Puiseux type transforms.

\subsection{Approximate roots and main result}\label{ssec:approx}

\paragraph*{Approximate roots.} The approximate roots of a monic
polynomial $F$ are defined thanks to the following proposition:

\begin{prop}\label{prop:approx-root} $($see e.g. {\upshape
    \cite[Proposition 3.1]{Pop02}).} Let $F\in\Ai[y]$ be monic of
  degree $\dy$, with $\Ai$ a ring whose characteristic does not divide
  $\dy$. Let $N\in \Ni$ dividing $\dy$. There exists a unique
  polynomial $\psi\in \Ai[y]$ monic of degree $d/N$ such that
  $\deg(F-\psi^{N}) < d-d/N$. We call it the $N^{th}$ approximate
  roots of $F$.
\end{prop}

A simple degree argument implies that $\psi$ is the
$N^{th}$-approximate root of $F$ if and only if the $\psi$-adic
expansion $\sum_{i=0}^N a_i \psi^i$ of $F$ satisfies $a_{N-1}=0$. For
instance, if $F=\sum_{i=0}^\dy a_i y^i$, the $\dy^{th}$ approximate
root coincides with the Tschirnhausen transform of $y$
\[
  \tau_F(y)=y+\frac{a_{\dy-1}}{\dy}.
\]
More generally, the $N^{th}$ approximate root can be constructed as
follows. Given $\phi\in \Ai[y]$ monic of degree $\dy/N$ and given
$F=\sum_{i=0}^N a_i \phi^i$ the $\phi$-adic expansion of $F$, we
consider the new polynomial
\[
  \tau_F(\phi):=\phi+\frac{a_{N-1}}{N}
\]
which is again monic of degree $\dy/N$. It can be shown that the
resulting $\tau_F(\phi)$-adic expansion $F=\sum a_i' \tau_F(\phi)^i$
satisfies $\deg(a_{N-1}')< \deg(a_{N-1})< \dy/N$ (see e.g. \cite[Proof
of Proposition 6.3]{Pop02}). Hence, after applying at most $\dy/N$
times the operator $\tau_F$, the coefficient $a_{N-1}'$ vanishes and
the polynomial $\tau_F\circ\cdots \circ \tau_F(\phi)$ coincides with
the approximate root $\psi$ of $F$. Although this is not the best
strategy from a complexity point of view (see Section \ref{sec:comp}),
this construction will be used to prove Theorem \ref{thm:HPsiPhi}
below.

\paragraph{Main result.} We still consider $F\in \Ki[[x]][y]$
Weierstrass of degree $\dy$ and keep notations from Section
\ref{sec:ARNP}. We denote $\psi_{-1}:=x$ and, for all $k=0,\ldots,g$,
we denote $\psi_k$ the $N_k^{th}$-approximate root of $F$. Fixing
$0\leq k\leq g$, we denote $\Psi=(\psi_{-1},\psi_0,\ldots,\psi_{k})$,
omitting once again the index $k$ for readibility.

Since $\deg \Psi=\deg \Phi$ by definition, the exponents of the
$\Psi$-adic expansion
\[
  F=\sum_{B\in \Bc} f_B' \Psi^B, \quad f_B'\in \Ki
\]
take their values in the same set $\Bc$ introduced in
\eqref{eq:Bc}. In the following, we denote by $w_i'\in\Ni$ the new
integer defined by \eqref{eq:wj} when replacing $f_B$ by $f'_B$ and we
denote $\bar{H}_k'$ the new polynomial obtained when replacing $w_i$
by $w_i'$ and $f_B$ by $f_B'$ in \eqref{eq:barHk}.

\begin{thm}\label{thm:HPsiPhi}
  We have $\bar{H}_k=\bar{H}'_k$ for $0\leq k< g$ and the \edgepoly{}s
  $\bar{H}_g$ and $\bar{H}'_g$ have same restriction to their Newton
  polygon's lowest edge.
\end{thm}

In other words, Theorems \ref{thm:NPphi} and \ref{thm:EdgePoly}
hold when replacing minimal polynomials by approximate roots, up to a
minor difference when $k=g$ that has no impact for degeneracy tests.

\paragraph{Intermediate results.} The proof of Theorem
\ref{thm:HPsiPhi} requires several steps.  We denote by
$-m_{g+1}/q_{g+1}$ the slope of the lowest edge of $H_g$.
\begin{lem}\label{lem:vpsiphi}
  We have $v_k(\psi_k-\phi_k) > v_k (\phi_k)+m_{k+1}/q_{k+1}$ for all
  $k=0,\ldots,g$.
\end{lem}

\begin{proof}
  Let $(q,m)=(q_{k+1},m_{k+1})$. The lemma is true if $\psi_k=\phi_k$
  and $\psi_k$ is obtained after successive applications of the
  operator $\tau_F$ to $\phi_k$. It is thus sufficient to prove
  \[
    v_k(\phi-\phi_k)> v_k (\phi_k)+m/q \Longrightarrow
    v_k(\tau_F(\phi)-\phi_k) > v_k (\phi_k)+m/q
  \]
  for any $\phi\in \Ki[[x]][y]$ monic of degree $d_k$. Suppose given
  such a $\phi$ and consider the $\phi$-adic expansion
  $F=\sum_{j=0}^{N_k} a_j \phi^{j}$. Then this implication holds
  if and only if
  \begin{equation}\label{eq:vkaNk}
    v_k(a_{N_k-1})>v_k (\phi_k)+m/q.
  \end{equation}  
  \noindent
  $\bullet$ \emph{Case $\phi=\phi_k$.} As $\phi_0=\psi_0=y+c_0$, we do
  not need to consider the case $k=0$. Let $k\ge 1$. Theorem
  \ref{thm:NPphi} and Lemma \ref{lem:FPhi} give
  $v_k(a_{N_k-1})\geq v_k(F)+m/q=N_k\,v_k(\phi_k)+m/q$. Note that
  $v_k(\phi_k)>0$ when $k\ge 1$ by construction. We are thus done when
  $N_k>1$. But $N_k=1$ means $k=g$ and $H_g=y$, so that
  $v_g(a_0)=\infty$. The claim follows.

  \noindent
  $\bullet$ \emph{Case $\phi\neq \phi_k$.} First note that
  $v_k(\phi-\phi_k) > v_k(\phi_k)$ implies $v_k(\phi) =v_k(\phi_k)$.
  As $\deg(\phi-\phi_k)<d_k$, we deduce from Corollary \ref{cor:main}
  (applied to $G=\phi-\phi_k$ and $i=0$) and Lemma \ref{lem:PikPhi}
  that
  \[
    \pi_k^*(\phi)=\pi_k^*(\phi-\phi_k)+\pi_k^*(\phi_k)=x^{v_k(\phi)}\,(y
    + x^{\alpha} \Ut)\,U_{k,k}
  \]
  where $\alpha:=v_k(\phi-\phi_k)-v_k(\phi_k)>m/q$ (hypothesis) and
  for some unit $U\in \Ki_k[[x,y]]^\times$. As $a_i$ has also
  degree $<d_k$, we deduce again from Corollary \ref{cor:main} that
  when $a_i\neq 0$,
  \begin{equation}\label{eq:ajphij}
    \pi_k^*(a_i\phi^i)=x^{\alpha_i}\,(y + x^{\alpha} U)^i\,U_i ,
  \end{equation}
  where $\alpha_i:=v_k(a_i\phi^i)$ and $U_i\in \Ki[[x,y]]^\times$.  As
  $\alpha>m/q$, this means that the lowest line with slope $-q/m$
  which intersects the support of $\pi_k^* (a_i\phi^{i})$ intersects
  it at the unique point $(i,\alpha_i)$. Since
  $ \pi_k^*(F)=\sum_{i=0}^{N_k} \pi_k^* (a_i\phi^{i})$, we deduce that
  the edge of slope $-q/m$ of the Newton polygon of $\pi_k^*(F)$
  coincides with the edge of slope $-q/m$ of the lower convex hull of
  $\left((i,\alpha_i)\;; a_i\neq 0, 0\leq i\leq N_k\right)$. 
    Thanks to
  \eqref{eq:pikHk} combined with $v_k(F)=N_k v_{k}(\phi_k)$ (Lemma
  \ref{lem:FPhi}) and $v_k(\phi_k)=v_k(\phi)$ (hypothesis), we deduce
  that the lowest edge $\Delta$ of $H_k$ (with slope $-q/m$) coincides
  with the edge of slope $-q/m$ of the lower convex hull of the points
  $\left((i,v_k(a_i)+(i-N_k)\,v_k(\phi))\;; a_i\neq 0, 0\leq i\leq
    N_k\right)$.  Since $H_k$ is monic of degree $N_k$ \textit{with no
    terms of degree $N_{k}-1$}, we deduce that $(N_k,0)\in \Delta$
  while $(N_{k}-1,v_k(a_{N_k-1})-v_k(\phi))$ must lie above
  $\Delta$. It follows that
  $ m N_k < m(N_{k}-1)+q (v_k(a_{N_k-1})-v_k(\phi)), $ leading to the
  required inequality $v_k(a_{N_k-1}) > v_k(\phi)+m/q$.  The lemma is
  proved.
\end{proof}
\begin{prop}\label{prop:psi}
  We have $v_k(\Psi)=v_k(\Phi)$ and $\lambda_k(\Psi)=\lambda_k(\Phi)$
  for all $k=0,\ldots,g$.
\end{prop}

\begin{proof}
  We show this result by induction. If $k=0$, we are done since
  $\psi_0=\tau_F(y)=\phi_0$. Let us fix $1\le k\le g$ and assume that
  Proposition \ref{prop:psi} holds for all $k'< k$. We need to
  show that $v_k(\psi_i)=v_{k}(\phi_i)$ and
  $\lambda_k(\psi_i)=\lambda_{k}(\phi_i)$ for all $i\le k$. Case $i=k$
  is a direct consequence of Lemma \ref{lem:vpsiphi}. For $i=k-1$,
  there is nothing to prove if $\phi_{k-1}=\psi_{k-1}$. Otherwise,
  using the linearity of $\pi_{k-1}^*$, Corollary \ref{cor:main}
  (applied at rank $k-1$ with $G=\phi_{k-1}-\psi_{k-1}$ and $i=0$) and
  Lemma \ref{lem:vpsiphi} give
  $ \pi_{k-1}^*(\psi_{k-1})=\pi_{k-1}^*(\phi_{k-1})+x^{\alpha} \Ut $
  with $\alpha>v_{k-1}(\phi_{k-1})+m_{k}/q_{k}$ and
  $\Ut\in \Ki_{k-1}[[x,y]]^\times$. We deduce
  $ \pi_k^*(\psi_{k-1})=\pi_k^*(\phi_{k-1})+x^{q_k\alpha}U_\alpha $
  with $U_\alpha\in \Ki_k[[x,y]]^\times$ and
  $q_k\,\alpha>v_{k}(\phi_{k-1})$ using Lemma \ref{lem:vLambda}
  ($q_k\,v_{k-1,k-1}+m_k=v_{k,k-1}$). This forces
  $v_k(\psi_{k-1})=v_k(\phi_{k-1})$ and
  $\lambda_k(\psi_{k-1})=\lambda_k(\phi_{k-1})$. Finally, for $i<k-1$,
  as $\deg(\psi_i)<d_{k-1}$, Corollary \ref{cor:main} (applied at rank
  $k-1$ with $G=\psi_i$ and $i=0$) gives
  \[
    \pi_{k-1}^*(\psi_i)=x^{v_{k-1}(\psi_i)}\lambda_{k-1}(\psi_i) U_i =
    x^{v_{k-1}(\phi_i)}\lambda_{k-1}(\phi_i) U_i ,
  \]
  where $U_i(0,0)=1$ (second equality by induction). Applying
  $\sigma_k^*$ and using Lemma \ref{lem:vLambda}, we conclude in the
  same way $v_k(\psi_{i})=v_k(\phi_{i})$ and
  $\lambda_k(\psi_{i})=\lambda_k(\phi_{i})$.
\end{proof}

\begin{cor}\label{cor:gprime}
  Let $G$ of degree less than $d_k$ and $\sum g'_B \Psi^B$ its
  $\Psi$-adic expansion. Then
  \[
    v_k(G)=\min (\langle B, V \rangle , g_B'\ne 0)\quad and \quad
    \lambda_k(G)=\sum_{B\in \Bc(0,v_k(G))} g'_B \Lambda^B.
  \]
  In particular, if $G$ has $\Phi$-adic expansion $\sum g_B \Phi^B$,
  then $g_B=g_B'$ when $\langle B, V \rangle=v_k(G)$.
\end{cor}

\begin{proof}
  As already shown in the proof of Proposition \ref{prop:psi}, from
  Corollary \ref{cor:main}, if $i<k$, we have
  $\pi_k^*(\psi_i)=x^{v_{k,i}}\lambda_{k,i}\,U_i$ with
  $U_i(0,0)=1$. As $\deg(G)< d_k$, we deduce
  \[
    \pi_k^*(G)=\sum g'_B \Lambda^B x^{\langle B, V \rangle} U_B
  \]
  with $U_B(0,0)=1$. This shows the result, using Proposition
  \ref{prop:key}.
\end{proof}

\paragraph{Proof of Theorem \ref{thm:HPsiPhi}.} Write
$F=\sum_i a_i\psi_k^i$ the $\psi_k$-adic expansion of $F$. Similarly
to \eqref{eq:ajphij}, when $a_i\neq 0$, Corollary \ref{cor:main} and
Lemma \ref{lem:vpsiphi} imply:
\begin{equation}
  \label{eq:ajpsij}
  \pi_k^*(a_i\psi_k^i)=x^{v_k(a_i\psi_k^i)}\, (y + x^{\alpha}\, \Ut)^i\,U,
\end{equation}
with $\alpha>m_{k+1}/q_{k+1}$, $U, \Ut\in \Ki_k[[x,y]]^\times$ and
$U(0,0)=\lambda_k(a_i\,\psi_k^i)$.  Applying the same argument than in
the proof of Lemma \ref{lem:vpsiphi}, we get that each point
$(i, w_i=N_k-i\,m_{k+1}/q_{k+1})$ of the lowest edge $\Delta$ of the
Newton polygon of $H_k$ (hence the whole polygon if $k<g$) is actually
$(i,v_k(a_i\,\psi_k^i)-v_k(F))$, that is $(i,w_i')$ from Corollary
\ref{cor:gprime} (applied to $G=a_i$) and Proposition
\ref{prop:psi}. This shows that we may replace $w_i$ by $w_i'$ in
\eqref{eq:wj}. More precisely, it follows from \eqref{eq:ajpsij} that
the restriction $\bar{H}_{k|\Delta}$ of $\bar{H}_k$ to $\Delta$ is
uniquely determined by the equality
\[
  \lambda_k(F)x^{v_k(F)}\bar{H}_{k|\Delta} = \sum_{(i,w_i')\in \Delta}
  \lambda_k(a_i\psi_k^i)x^{v_k(a_i\psi_k^i)} y^i.
\]
Using again Corollary \ref{cor:gprime} and Proposition \ref{prop:psi},
we get
\[
  \bar{H}_{k|\Delta} = \sum_{(i,w_i')\in \Delta}
  \left(\sum_{B\in\Bc(i,w_i'+v_k(F))} f_B' \Lambda^{B-B_0}\right)
  x^{w_i'} y^i,
\]
as required. \hfill$\square$

\begin{rem}
  Theorem \ref{thm:HPsiPhi} would still hold when replacing $\psi_k$
  by any monic polynomial $\phi$ of same degree for which
  $\pi_k^*(\phi)=x^{v_{k,k}}\,(y+\beta(x))\,U $ with
  $\val(\beta)> m_{k+1}/q_{k+1}$.
\end{rem}

\subsection{An Abhyankar type irreducibility test}

Theorem \ref{thm:HPsiPhi} leads to the following sketch of algorithm.
Subroutines \AppRoot{}, \Expand{} and \Edgepoly{} respectively compute
the approximate roots, the $\Psi$-adic expansion and the current lowest
\edgepoly{} (using \eqref{eq:wj} and \eqref{eq:barHk}). They are
detailed in Section \ref{sec:comp}. Also, considerations about
truncation bounds is postponed to Section \ref{ssec:trunc}. Given a
ring $\Li$ and $P\in \Li[Z]$, we denote by $\Li_P=\Li[Z]/(P(Z)$.

\begin{algorithm}[H]
  \nonl\TitleOfAlgo{\irreducible$(F, \Li)$\label{algo:irreducible}}%
  \KwIn{$F\in\Ki[[x]][y]$ Weierstrass with $\dy=\deg(F)$ not divisible
    by the characteristic of $\Ki$ ; $\Li$ a field extension of
    $\Ki$.}%
  \KwOut{\True{} if $F$ is irreducible in $\Li[[x]][y]$, and \False{}
    otherwise.}%
  $N\gets \dy$, $V\gets (1,0)$, $\Lambda\gets (1,1)$,
  $\Psi\gets(x)$\;%
  \While{$N > 1$}{%
    $\Psi\gets \Psi\cup \AppRoot{}(F,N)$\;%
    $\sum_{B} f_B \Psi^{B}\gets \Expand{}(F,\Psi)$\;%
    $\bar{H}\gets \Edgepoly{}(F,\Psi,V,\Lambda)$\;%
    \lIf{$\bar{H}$ is not degenerated over $\Li$}{\Return{\False}}%
    $(q,m,P,N)\gets \Edata(\bar{H})$\;%
    Update the lists $V,\Lambda$ using \eqref{eq:update}\;%
    $\Li\gets \Li_P$ }%
  \Return{\True}
\end{algorithm}

\begin{thm}\label{thm:Irr}
  Algorithm \irreducible{} returns the correct answer.
\end{thm}

\begin{proof}
  This follows from Theorem \ref{thm:NPphi}, \ref{thm:EdgePoly} and
  \ref{thm:HPsiPhi}, together with Proposition \ref{prop:NPA}.
\end{proof}

Let us illustrate this algorithm on two simple examples.

\begin{xmp}\label{exKuo}
  Let $F(x,y)=(y^2-x^3)^2-x^7$. This example was suggested by Kuo who
  asked if we could show that $F$ is reducible in
  $\algclos{\Qi}[[x]][y]$ without performing Newton-Puiseux type
  tranforms. Abhyankhar solved this challenge in \cite{Ab89} thanks to
  approximate roots. Let us show that we can prove further that $F$ is
  reducible in $\Qi[[x]][y]$ without performing Newton-Puiseux type
  tranforms.

  \textit{Initialisation.} Start from $\psi_{-1}=x$, $N_0=\dy=4$,
  $V=(1,0)$ and $\Lambda=(1,1)$.

  \textit{Step k=0.} The $4^{th}$ approximate root of $F$ is
  $\psi_{0}=y$. So $H_0=F$ and we deduce from \eqref{eq:barHk} (see
  Exemple \ref{xmp:k0}) that $\bar{H}_0=(y^2-x^3)^2$. Hence, $F$ is
  degenerated with edge data $(q_1,m_1,P_1,N_1)=(2,3,Z_1-1,2)$ and we
  update $V=(2,3,6)$ and $\Lambda=(1,1,2)$ thanks to
  \eqref{eq:update}, using here $z_1=1\mod P_1$.

  \textit{Step k=1.} The $2^{nd}$ approximate root of $F$ is
  $\psi_1=y^2-x^3$ and $F$ has $\Psi$-adic expansion
  $F=\psi_1^2-\psi_{-1}^7$. We have $v_1(\psi_1^2)=2v_{1,1}=12$,
  $\lambda_1(\psi_1^2)=\lambda_{1,1}^2=4$ while
  $v_1(\psi_{-1}^7)=7v_{-1,1}=14$ and
  $\lambda_1(\psi_{-1}^7)=\lambda_{-1,1}^7=1$. We deduce from
  \eqref{eq:barHk} that $\bar{H}_1=y^2-\frac{1}{4} x^2$. As the
  polynomial $Z_2^2-\frac{1}{4}$ is reducible in
  $\Qi_{P_1}[Z_2]=\Qi[Z_2]$, we deduce that $F$ is reducible in
  $\Qi[[x]][y]$.
\end{xmp}

\begin{xmp}\label{ex:2}
  Consider $F=((y^2-x^3)^2+4 x^8)^2+x^{14}(y^2-x^3)$ (we
  assume that we only know its expanded form at first).

  \textit{Initialisation.} We start with $\psi_{-1}=x$, $N_0=\dy=8$,
  $V=(1,0)$ and $\Lambda=(1,1)$.

  \textit{Step k=0.} The $8^{th}$ approximate root of $F$ is
  $\psi_{0}=y$.  The monomials reaching the minimal values
  \eqref{eq:wj} in the $\Psi=(\psi_{-1},\psi_0)$-adic expansion of $F$
  are $\psi_0^8$, $-4\psi_{-1}^3 \psi_0^6$, $6\psi_{-1}^6 \psi_0^4$, $-4\psi_{-1}^9\psi_0^2$, $\psi_{-1}^{12}$ and we deduce from
  \eqref{eq:barHk} that $\bar{H}_0=(y^2-x^3)^4$. Hence,
  $(q_1,m_1,P_1,N_1)=(2,3,Z_1-1,4)$ and we update $V=(2,3,6)$ and
  $\Lambda=(1,1,2)$ thanks to \eqref{eq:update}, using here
  $z_1=1\mod P_1$.

  \textit{Step k=1.} The $4^{th}$ approximate root of $F$ is
  $\psi_1=y^2-x^3$ and we get the current $\Psi$-adic expansion
  $F=\psi_1^4+8\psi_{-1}^8 \psi_1^2+\psi_{-1}^{14}
  \psi_1+16\psi_{-1}^{16}$.  The monomials reaching the minimal values
  \eqref{eq:wj} are $\psi_1^4$, $8\psi_{-1}^8 \psi_1^2$,
  $16\psi_{-1}^{16}$ and we deduce from \eqref{eq:barHk} that
  $\bar{H}_1=(y^2+x^4)^2$. Hence $(q_2,m_2,P_2,N_2)=(1,2,Z_2^2+1,2)$
  and we update $V=(2,3,8,16)$ and $\Lambda=(1,1,2z_2,8 z_2)$ thanks
  to \eqref{eq:update}, where $z_2=Z_2\mod P_2$ and using the B\'ezout
  relation $q_2 s_2-m_2 t_2=1$ with $(s_2,t_2)=(1,0)$. Note that we
  know at this point that $F$ is reducible in $\algclos{\Qi}[[x]][y]$
  since $P_2$ has two distinct roots in $\algclos{\Qi}$.

  \textit{Step k=2.} The $2^{nd}$ approximate roots of $F$ is
  $\psi_2=(y^2-x^3)^2+4x^8$ and we get the current $\Psi$-adic
  expansion $F=\psi_2^2+\psi_{-1}^{14}\psi_1$.  The monomials reaching
  the minimal values \eqref{eq:wj} are $\psi_2^2$,
  $\psi_{-1}^{14}\psi_1$ and we deduce from \eqref{eq:barHk} that
  $\bar{H}_2=y^2+(32 z_2)^{-1} x$ (note that $z_2$ is invertible in
  $\Qi_{P_2}$).  Hence $\bar{H}_2$ is degenerated with edge data
  $(q_3,m_3,P_3,N_3)=(2,1,Z_3+(32 z_2)^{-1},1)$. As $N_3=1$, we deduce
  that $F$ is irreducible in $\Qi[[x]][y]$ ($g=3$ here).
\end{xmp}

\begin{rem}\label{rem}
  Note that for $k\ge 2$, we really need to consider the $\Psi$-adic
  expansion: the $(x,y,\psi_k)$-adic expansion is not enough to
  compute the next data. At step $k=2$ in the previous example, the
  $\psi_2$-adic expansion of $F$ is $F=\psi_2^2+a$ where
  $a=x^{14} y^2-x^{17}$. We need to compute $\val[2](a)$. Using the
  $\Psi$-adic expansion $a=\psi_{-1}^{14} \psi_1$, we find
  $\val[2](a)=14 \times 2+ 8 = 36$.  Considering the $(x,y)$-adic
  expansion of $a$ would have led to the wrong value
  $\val[2](x^{14} y^2)= \val[2](x^{17}) =34 < 36$.
\end{rem}

\subsection{Quasi-irreducibility}\label{ssec:PseudoIrr}
In order to perform a unique irreducibility test, we will rather relax
the degeneracy condition by allowing square-freeness of the involved
residual polynomial $P_1,\ldots,P_g$, and eventually check if $\Ki_g$
is a field. This leads to what we call a quasi-irreducibility test.
The fields $\Ki_k$'s become ring extensions of $\Ki$ isomorphic to a
direct product of fields and we have to take care of zero divisors.

Let $\Ai=\Li_0\oplus \cdots \oplus \Li_r$ be a direct product of
perfect fields. We say that a polynomial $H$ defined over $\Ai$ is
\textit{square-free} if its projections under the natural morphisms
$\Ai\to \Li_i$ are square-free (in the usual sense over a field). If
the polynomial is univariate and monic, this exactly means that its
discriminant is not a zero divisor in $\Ai$.

\begin{dfn}\label{def:quasideg}
  We say that a Weierstrass polynomial $F\in \Ai[[x]][y]$ is
  \emph{quasi-degenerated} if its \edgepoly{} has shape
  \[
    \bar{F}= \left(P\left(\frac{y^q}{x^{m}}\right)x^{m\deg(P)}\right)^N
  \]
  with $q,m$ coprime and $P\in \Ai[Z]$ monic, \emph{square-free} and
  satisfying $P(0)\in \Ai^\times$.
\end{dfn}

We abusively still call $P$ the \emph{residual polynomial} of $F$ and
$(q,m,P,N)$ the \emph{edge data} of $F$, with convention $(q,m)=(1,0)$
if the Newton polygon is reduced to a point.
  
\begin{dfn}\label{def:quasi}
  We call \PIrr{} the new algorithm obtained when replacing
  degenerated tests by quasi-degenerated tests in algorithm
  \irreducible. $F\in \Ki[[x]][y]$ Weierstrass is said
  quasi-irreducible over $\Li$ if \PIrr($F,\Li$) outputs \True{}.
\end{dfn}

As $P_k(0)\in \Ki_{k-1}^\times$ by assumption, $z_k$ is not a zero
divisor in $\Ki_k$. It follows straightforwardly that all results of
Section \ref{sec:phi} and Subsection \ref{ssec:approx} still hold when
considering quasi-degeneracy. In particular, algorithm \PIrr{} is
well-defined and Definition \ref{def:quasi} makes sense.

\begin{lem}\label{lem:quasiIrrvsIrr}
  A square-free monic polynomial $F\in \Ki[[x]][y]$ is irreducible
  over $\Ki$ if and only if it is quasi-irreducible and $\Ki_g$ is a
  field.
\end{lem}

\begin{proof} This follows immediately from Definitions
  \ref{def:quasideg} and \ref{def:quasi} with Theorem
  \ref{thm:Irr}.
\end{proof}


\section{Complexity. Proof of Theorem \ref{thm:main}}\label{sec:comp}

\subsection{Complexity model}
\label{ssec:model}
We use the algebraic RAM model of Kaltofen \cite[Section 2]{Ka88},
counting only the number of arithmetic operations in our base field
$\Ki$.  Most subroutines are deterministic; for them, we consider the
worst case. However, computation of primitive elements in residue
fields uses a probabilistic algorithm of Las Vegas type, and we
consider then the average running time. We denote by $\M(d)$ the
number of arithmetic operations for multiplying two polynomials of
degree $d$. We use fast multiplication, so that $\M(d)\in \Ot(d)$ and
$d'\M(d)\le \M(d'd)$, see \cite[Section 8.3]{GaGe13}.  We use the
classical notations $\O()$ and $\Ot()$ that respectively hide constant
and logarithmic factors \cite[Chapter 25, Section 7]{GaGe13}.

$F$ being Weierstrass, we have the following result. As $\vF>0$, that
ensures in particular that e.g. $\vF\,\log(\dy)\in\Ot(\vF)$ (this will
be used several times in the following).
\begin{lem}\label{lem:d-delta}
 We have $\dy-1\leq\vF$. If $F$ is
  quasi-irreducible with residual degree $f$, then $(\dy-1)\,f\leq \vF$.
\end{lem}
\begin{proof}
  As $F$ is Weierstrass, all its Puiseux series have valuation at
  least $1/\dy$. Seeing $\vF$ as the sum of the valuations of the
  difference of the Puiseux series of $F$ concludes the first
  point. In the second case, the minimum valuation for the Puiseux
  series of $F$ becomes $f/\dy$ (just use the classical equality
  $\dy=e\,f$).
\end{proof}
\paragraph{Primitive representation of residue rings.}
The $\Ki$-algebra $\Ki_k$ is given inductively as a tower extension of
$\Ki$ defined by the radical triangular ideal
$(P_1(Z_1),\ldots,P_k(Z_1,\ldots,Z_k))$. It turns out that such a
representation does not allow to reduce a basic operation in $\Ki_k$
to $\Ot(f_k)$ operations over $\Ki$ (see \cite{PoWe17} for
details). To solve this problem, we compute a primitive representation
of $\Ki_k$, introducing the notation $\Ki_Q:=\Ki[T]/(Q(T))$.

\begin{prop}\label{prop:Prim}
  Let $Q\in \Ki[T]$ and $P\in \Ki_Q[Z]$ square-free, and assume that
  $\Ki$ has at least $(\deg_T(Q)\,\deg_Z(P))^2$ elements.  There
  exists a Las Vegas algorithm \Primitive{} that returns $(Q_1,\tau)$
  with $Q_1\in \Ki[W]$ square-free and
  $\tau:\Ki[T,Z]/(Q,P) \rightarrow \Ki[W]/(Q_1)$ an isomorphism. It
  takes an expected $\O((\deg_T(Q)\,\deg_Z(P))^{(\omega+1)/2})$
  operations over $\Ki$. Given $\alpha\in\Ki[T,Z]/(Q,P)$, one can
  compute $\tau(\alpha)$ in less than $\Ot(\deg_T(Q)^2\,\deg_Z(P))$.
\end{prop}

\begin{proof}
  Use \cite[Proposition 15]{PoWe17} with $I=(Z_1,Q(Z_2))$ (see
  notations therein).
\end{proof}

In the following, we use that an operation in $\Ki_k$ costs $\Ot(f_k)$
operations in $\Ki$.
\begin{rem}
  Another way to deal with tower extensions would be the recent
  preprint \cite{HoLe18}. This would make all algorithms
  deterministic, with a cost $\O(\vF^{1+o(1)})$ instead of $\Ot(\vF)$.
  Note also \cite{HoLe19} for dynamic evaluation.
\end{rem}

\subsection{Truncation bounds}
\label{ssec:trunc}

In order to estimate the complexity in terms of arithmetic operations
in $\Ki$, we will compute approximate roots and $\Psi$-adic expansions
modulo a suitable truncation bound for the powers of $\psi_{-1}=x$. We
show here that the required sharp precision is the same than the one
obtained in \cite[Section 3]{PoWe17} for the Newton-Puiseux type
algorithm. Note also \cite[Theorem 2.3, page 144]{BaNaSt13} that
provides similar results in the context of irreducibility test. In the
following, when we say that we truncate a polynomial with precision
$\tau\in\Qi$, we mean that we keep only powers of $X$ less or equal
than $\tau$.

The successive polynomials generated by the function call \PIrr{}($F$)
are still denoted $H_0,\ldots,H_g$, and we let $(q_{g+1},m_{g+1})$
stand for the slope of the lowest edge of $H_g$, with convention
$(q_{g+1},m_{g+1})=(1,0)$ if $N_g=1$. As $\deg(H_k)=N_k$ and
$\NP(H_k)$ has a lowest edge of slope $-m_{k+1}/q_{k+1}$, the
computation of the lowest \edgepoly{} $\bar{H}_k$ only depends on $H_k$
truncated with precision $N_k m_{k+1}/q_{k+1}$.  Combined with
\eqref{eq:pikHk}, and using $\val(\pi_k^*(x))=e_k$, we deduce that the
$k^{th}$-edge data only depends on $F$ truncated with precision
\[
  \eta_k:=\frac{v_k(F)}{e_k}+ N_k \frac{m_{k+1}}{e_{k+1}}.
\]
Denoting $\displaystyle{} \eta(F):=\max_{0\leq k\leq g}(\eta_k)$, we
deduce that running \PIrr{} modulo $x^{\eta(F)+1}$ returns the correct
answer, this bound being sharp by construction.

\begin{lem}\label{lem:etaF}
  Denoting $\eta_{-1}=0$, we have
  $\eta_{k}=\eta_{k-1}+\frac{N_k m_{k+1}}{e_{k+1}}$ for any
  $0\leq k\leq g$. In particular,
  $\eta (F)= \eta_g=\sum_{k=1}^{g+1} \frac{N_{k-1} m_k}{e_k}$.
\end{lem}

\begin{proof}
  As $v_k(F)=N_k\, v_{k,k}$ from Lemma \ref{lem:FPhi}, we get for any $0\leq k\leq g$
  \[
    \eta_k=\frac{N_{k} v_{k,k}}{e_{k}}+\frac{N_{k} m_{k+1}}{e_{k+1}}.
  \]
  As $v_{0,0}=0$, case $k=0$ is proved. Let $k\ge 1$. Previous formula used at rank $k-1$ gives
  \[
    \eta_{k-1}=\frac{N_{k-1} v_{k,k-1}}{e_{k}}=\frac{N_{k} v_{k,k}}{e_{k}},
  \]
  first equality using Point 1 of Lemma \ref{lem:vLambda}
  ($v_{k,k-1}=q_k v_{k-1,k-1}+m_k$) and second equality using
  $N_{k-1}= q_{k} \ell_{k} N_{k}$ and equality
  $v_{k,k}=q_{k} \ell_{k} v_{k,k-1}$ of Proposition
  \ref{prop:vkk}. This gives
  $\eta_k = \eta_{k-1}+\frac{N_k m_{k+1}}{e_{k+1}}$ as required. The
  formula for $\eta(F)$ follows straightforwardly.
\end{proof}

\begin{rem} We have the formula
  $\eta_k = \frac{\val(\pi_k^*F(x,0))}{e_{k+1}} =
  \frac{(F,\,\phi_k)_0}{d_k}$ for $k< g$, from respectively
  \eqref{eq:pikHk} and Corollary \ref{cor:res}. We deduce in
  particular that the sequence
  $(N_0,d_0\eta_0,\ldots,d_{g-1} \eta_{g-1})$ is a minimal set of
  generators of the semi-group of $F$ when $F$ is irreducible in
  $\algclos{\Ki}[[x]][y]$ ; see e.g. \cite[Proposition 4.2 and Theorem
  5.1]{Pop02}.
\end{rem}

\begin{prop}\label{prop:eta}
  Let $F\in\Ki[[x]][y]$ be monic and separable of degree $\dy$, with
  discriminant valuation $\vF$. Then $\eta(F) \le \frac{2\vF}{\dy}$.
  If moreover $F$ is quasi-irreducible, then $\eta(F)\ge \vF/\dy$.
\end{prop}

\begin{proof}
  It follows from Lemma \ref{lem:etaF} that $\eta(F)$ is smaller or
  equal than the quantity ``$N_i$'' defined in \cite[Subsection
  3.3]{PoWe17} (take care of notations, these $N_i$ are not the same
  as those defined here), with equality if $F$ is
  quasi-irreducible. From \cite[Corollary 4]{PoWe17}, we deduce
  $\eta(F) \le 2 v_i$ for $i=1,\ldots,\dy$, where
  $v_i:=\val(\partial_y F(y_i))$, $y_i$ denoting the roots of $F$. As
  $\vF=\sum v_i$, we have $\min v_i\le \vF/\dy$ and the upper bound
  for $\eta(F)$ follows. If $F$ is quasi-irreducible, then we have
  also $v_i \le \eta(F)=N_i$ by \cite[Corollary 4]{PoWe17}. As all
  $v_i$'s are equal in that case, the lower bound follows too.
\end{proof}

\begin{rem}[\emph{Dealing with the precision}]\label{rem:precision} As
  $\vF$ is not given, we do not have an \textit{a piori} bound for the
  precision $\eta(F)$. To deal with this problem, one can either use
  \emph{relaxed} computations \cite{Ho02} or just restart the whole
  computation when we realise that we are missing precision. With both
  solutions, we need to increase the precision each time the computed
  lowest edge of the Newton polygon is not ``guaranted'' in the sense
  of \cite[Definition 8 and Figure 1.b]{PoWe17}.  In algorithm \PIrr{}
  below, we use the second option; this is done at lines
  \ref{Pirr:etaprim} and \ref{Pirr:callback}, thanks to Lemma
  \ref{lem:etaF}. In terms of complexity, both solutions only
  multiply the complexity bound by a logarithmic factor.
\end{rem}

\subsection{Main subroutines}
\label{ssec:subroutines}

\paragraph{Computing approximate roots and $\Psi$-adic expansion.}
\begin{prop}\label{prop:approx}
  There exists an algorithm $\AppRoot{}$ which given $F\in \Ai[y]$ a
  degree $d$ monic polynomial defined over a ring of characteristic
  not dividing $d$ and given $N$ which divides $d$, returns the
  $N^{th}$ approximate root $\psi$ of $F$ with $\M(d)$ operations over
  $\Ai$.
\end{prop}

\begin{proof}
  Let $G=y^{d}F(1/y)$ be the reciprocal polynomial of $F$. So $G(0)=1$
  and there exists a unique series $S\in \Ai[[y]]$ such that $S(0)=1$
  and $G=S^N$. Then $\psi$ is the reciprocal polynomial of the
  truncated series $\tronc{S}{\frac d N}$ (see e.g. \cite[Proposition
  3.4]{Pop02}). The serie $S$ is solution of the equation $Z^N-G=0$ in
  $\Ai[[y]][Z]$ and can be computed up to precision $d/N$ within
  $\M(d)$ operations by quadratic Newton iteration \cite[Theorem
  9.25]{GaGe13}.
\end{proof}

\begin{prop}\label{prop:expand}
  There exists an algorihm \Expand{} which, given $F\in \Ai[y]$ of
  degree $d$ and $\Psi=(\psi_{0},\ldots,\psi_k)$ a collection of monic
  polynomials $\psi_i\in \Ai[y]$ of strictly increasing degrees
  $d_0<\cdots < d_k\le d$ returns the reduced $\Psi$-adic expansion of
  $F$ in less than $\O((k+1)\,\M(d)\,\log(d))$ arithmetic operations over $\Ai$.
\end{prop}

\begin{proof}
  The $\psi_k$-adic expansion of $F=\sum a_i \psi_k^i$ requires
  $\O(\M(d)\log(d))$ operations by \cite[Theorem 9.15]{GaGe13}. If
  $k>0$, we recursively compute the $(\psi_0,\ldots,\phi_{k-1})$-adic
  expansion of $a_i$ in $\O(k\M(\deg\,a_i)\log(\deg\,a_i))$
  operations. Since $\deg(a_i)< d_k$, summing over all
  $i=0,\ldots,\lfloor d/d_k\rfloor $ gives $\O(k\,\M(d)\,\log(d))$
  operations.
\end{proof}

\paragraph{Computing \edgepoly{}s.}
\begin{prop}\label{prop:computeHbark}
  Given $F$ and $\Psi=(\psi_{-1},\ldots,\psi_{k})$ modulo
  $x^{\eta(F)+1}$, $V=(v_{k,-1},\ldots,v_{k,k})$ and
  $\Lambda=(\lambda_{k,-1},\ldots,\lambda_{k,k})$, there exists an
  algorithm \Edgepoly{} that computes the lowest \edgepoly{}
  $\bar{H}_k\in \Ki_k[x,y]$ within $\Ot(\vF+f_k^2)$ operations over
  $\Ki$.
\end{prop}

\begin{proof}
  First compute the $\Psi$-adic expansion $F=\sum f_B \Psi^B$ modulo
  $x^{\eta+1}$, with $\eta:=\eta(F)$. As $\eta\le 2\vF/\dy$, this is
  $\Ot(\vF)$ by Proposition \ref{prop:expand} applied with
  $\Ai=\Ki[x]/(x^{\eta+1})$, using $k\leq\log(\dy)$ and Lemma
  \ref{lem:d-delta}. We compute the lowest edge of $\NP(H_k)$ via
  Theorem \ref{thm:NPphi}; this takes no arithmetic
  operations\footnote{for the interested reader, it can easily be
    shown that this takes $\Ot(\vF)$ bit operations}. It remains to
  compute the coefficient of each monomial $x^{w_i} y^i$ of
  $\bar{H}_k$, which is
  \[
    c_{k,i}:=\sum_{B\in \Bc(i,w_i+v_k(F))} f_B \Lambda^{B-B_0}
  \]
  by Theorem \ref{thm:EdgePoly}. The computation of
  $\Lambda^{B_0}=\lambda_{k,k}^{N_k}$ takes $\O(\log(d))$ operations
  over $\Ki_k$ via fast exponentiation. Also, there are at most $f_k$
  monomials $\Lambda^B$ to compute from Proposition
  \ref{prop:key}. Each of them can be computed in $\O(k\,\log(\vF))$
  operations in $\Ki_k$ via fast exponentiation on each
  $\lambda_{k,i}$ (we have
  $w_i\le \val(H_k(x,0))= N_{k}m_{k+1}/q_{k+1}$, thus
  $w_i+v_k(F)\le e_k\eta_k\le 2\vF$ by definition of $\eta_k$ and
  Proposition \ref{prop:eta}). This concludes.
\end{proof}

\paragraph{Testing quasi-degeneracy and computing edge data.}
\begin{prop}\label{prop:quasideg}
  Given $Q\in \Ki[Z]$ square-free and $\bar{H}\in \Ki_Q[x,y]$ monic in
  $y$ and quasi-homogeneous, there exists an algorithm \PDegenerated{}
  that returns \False{} if $\bar H$ is not quasi-degenerated, and the
  edge data $(q,m,P,N)$ of $\bar H$ otherwise. It takes at most
  $\Ot(\deg_Z(Q) \deg_y(\bar H)/q)$ operations over $\Ki$.
\end{prop}

\begin{proof} 
  As $\bar H$ is quasi-homogeneous, we have
  $\bar H=P_0(y^q/x^m)x^{m\deg(P_0)}$ for some coprime integers
  $q,m\in \Ni$ and some $P_0\in \Ki_{Q}[T]$ of degree
  $\deg_y(\bar H)/q$. We need to check if $P_0=P^{N}$ for some $N\in \Ni$
  and $P\in\Ki_Q[T]$ square-free $($i.e $(Q,P)$ radical ideal in
  $\Ki[Z,T])$, and that $P(0)\notin\Ki_Q^\times$. The first task is a
  special case of \cite[Proposition 14]{PoWe17} and fits in the aimed
  bound. Second one is just a gcd computation, bounded by
  $\Ot(\deg_Z(Q))$.
\end{proof}

\subsection{The main algorithm. Proof of Theorem
  \ref{thm:main}}\label{ssec:proofs}

\begin{algorithm}[H]
  \nonl\TitleOfAlgo{\PIrr{}($F,\eta=1$)\label{algo:PIrr}}%
  \KwIn{$F\in\Ki[[x]][y]$ Weierstrass of degree $\dy$ not divisible by
    $\Char(\Ki)$.}%
  \KwOut{\False{} if $F$ is not quasi-irreducible, and $(\D,Q)$
    otherwise, with $\D$ the edge data of $F$ and $\Ki_g=\Ki_Q$.}%
  $F\gets F\mod x^{\eta}$ \tcp*{All computations modulo $x^\eta$}%
  $N\gets \dy$, $V\gets [1,0]$, $\Lambda\gets [1,1]$, $\Psi\gets[x]$,
  $Q\gets Z$, $(e,\eta')\gets (1,0)$, $\D\gets[\,]$\;%
  \While{$N > 1$}{%
    $\Psi\gets \Psi\cup \AppRoot{}(F,N)$\label{Pirr:approx}\;%
    $\bar H\gets
    \Edgepoly{}(F,\Psi,V,\Lambda)$
    \label{Pirr:edge}\tcp*{$\bar H\in \Ki_Q[x,y]$}%
    $e\gets q\,e$ ; $\eta'\gets \eta'+ \frac{N\,
      m}{e}$ \label{Pirr:etaprim}\tcp*{$(q,m)$ the lowest edge of $\bar
      H$}%
    \lIf{$\eta \le\eta'$}{\Return
      \PIrr{}($F,2\eta$)}\label{Pirr:callback}%
    $(Bool,(q,m,P,N))\gets$ \PDegenerated$(\bar
    H,Q)$ \label{Pirr:degen}\;%
    \lIf{$Bool=$ \False}{\Return{\False}}%
    $\D\gets \D\cup (q,m,P,N)$\;%
    Update the lists $V,\Lambda$ using \eqref{eq:update}\;%
    $(Q,\tau)\gets$ \Primitive{}$(Q,P)$\label{Pirr:prim}\;%
    $\Lambda\gets \tau(\Lambda)$\label{Pirr:updatelambda}\;%
  }%
  \Return $(\D,Q)$\;%
\end{algorithm}

\begin{prop}\label{prop:fastIrr}
  Running \PIrr{}$(F)$ returns the correct output in an expected
  $\Ot(\vF)$ operations over $\Ki$.
\end{prop}

\begin{proof}
  The polynomial $\bar H$ at line \ref{Pirr:degen} is the correct
  lowest \edgepoly{} thanks to Lemma \ref{lem:etaF} (see also Remark
  \ref{rem:precision}). Then correctness follows from Theorem
  \ref{thm:Irr} and Definition \ref{def:quasi}. As $q_k\ell_k \ge 2$,
  we have $g\le \log_2(\dy)$, while recursive calls of line
  \ref{Pirr:callback} multiply the complexity by at most a logarithm
  too. Considering one iteration, and using
  $\eta< 2\,\eta(F)\leq 4\,\vF/\dy$ (second inequality by Proposition
  \ref{prop:eta}), lines \ref{Pirr:approx}, \ref{Pirr:edge},
  \ref{Pirr:degen}, \ref{Pirr:prim} and \ref{Pirr:updatelambda} cost
  respectively $\Ot(\vF)$, $\Ot(\vF+f_k^2)$,
  $\Ot(f_k N_k/q_k)\subset \Ot(\dy)$, $\Ot(f_k^{(\omega+1)/2})$ and
  $\Ot(f_k^2)$ from respectively Propositions \ref{prop:approx},
  \ref{prop:computeHbark}, \ref{prop:quasideg}, \ref{prop:Prim} and
  \ref{prop:Prim} once again. Summing up, we conclude from Lemma
  \ref{lem:d-delta} (note that if $F$ is not quasi-irreducible, as
  long as the algorithm does not output \False{}, $F$ ``looks''
  quasi-irreducible, so that we still have $(\dy-1)\,f_g\leq\vF$).
\end{proof}

\paragraph{Proof of Theorem \ref{thm:main}.} 
Thanks to Lemma \ref{lem:quasiIrrvsIrr}, $F$ is irreducible if and
only if it is quasi-irreducible and the residue ring
$\Ki_g=\Ki[Z]/(Q(Z))$ is a field. This can be checked with a
univariate irreducibility test in $\Ki[Z]$ of degree
$\deg(Q)=f\le \dy$. $\hfill\square$

Note that there are well known formulas for the valuation of the
discriminant $\vF$ in terms of the edge data, see
e.g. \cite[Corollary 5]{PoWe19}.

\begin{xmp}\label{xmp:splitting}
  Let us illustrate algorithm \PIrr{} on a simple example, considering
  $F=(y^4-x^2)^4+y^6x^{11}-y^4x^{12}-y^2x^{13} +x^{14}+x^{16}$ with $\Ki=\Qi$.

  \emph{Initialisation.} We start with $N_0=\dy=16$, $\psi_{-1}=x$,
  $V=(1,0)$ and $\lambda=(1,1)$.

  \emph{Step $0$.} The $16^{th}$-approximate roots of $F$ is
  $\psi_0=y$ and we find $\bar{H}_0=(y^4-x^2)^4$. So $H_0$ is
  quasi-degenerated with edge data
  $(q_1,m_1,P_1,N_1)=(2,1,Z_1^2-1,4)$. Using
  \eqref{eq:update}, we update $V=(2,1,4)$ and
  $\lambda=(z_1,z_1,4z_1)$, with $z_1=Z_1\mod P_1$ (i.e.  $z_1^2=1$).
  
  \emph{Step $1$.} We compute the $4^{th}$-approximate root
  $\psi_1=y^4-x^2$ of $F$, then its $\Psi$-adic expansion
  $F = \psi_1^4 + \psi_{-1}^{11} \psi_0^2 \psi_1 - \psi_{-1}^{12}
  \psi_1 + \psi_{-1}^{16}$.
  All involved monomials reach the minimal values \eqref{eq:wj}, and
  we deduce from \eqref{eq:barHk} and equality $z_1^2=1 $ that
  $\bar{H}_1=y^4+\frac{(1-z_1)}{4^3} x^{12}y+\frac{1}{4^4}x^{16}$,
  which is quasi-homogeneous with slope $(q_2,m_2)=(1,4)$. We find
  that $P_0=Z_2^4+ \frac{(1-z_1)}{4^3}Z_2+\frac{1}{4^4}$ is square-free
  over $\Qi_1$. Hence, $\bar{H}_1$ is quasi-degenerated with edge data
  $(q_2,m_2,P_2,N_2)=(1,4,P_0,1)$. As $N_2=1$, we deduce that $F$ is
  quasi-irreducible. However, the last residue field
  $\Qi_2=\Qi[Z_1,Z_2]/(P_1,P_2)$ is not a field so $F$ is not
  irreducible (in practice, the algorithm would have computed
  $Q\in \Qi[Z]$ of degree $8$ such that $\Qi_2=\Qi[Z]/(Q(Z))$, and
  eventually check the irreducibility of $Q$).
\end{xmp}

\begin{rem}\label{rem:sqrfree}
  The polynomial $Q$ might factor during the square-free test made at
  Line \ref{Pirr:degen}. In such a case, $F$ is reducible and we
  should of course immediately return \False{} at this stage. For
  instance testing square-freeness of $P_0$ in Example
  \ref{xmp:splitting} requires to compute the gcd between $P_0$ and
  its derivative $P_0'$. The first euclidean division gives
  $P_0=\frac{Z_2}{4}\,P_0'+R$ with
  $R=\frac 3{4^4}\,(1-z_1)\,Z_2+\frac 1 {4^4}$. Before proceeding to
  the next division of $P_0'$ by $R$, we need to check first that the
  leading coefficient of $R$ is a unit in $\Qi_1$.  To this aim, we
  compute the gcd between $\frac 3{4^4}\,(1-Z_1)$ and $P_1$,
  discovering here that $Z_1-1$ divides $P_1$ so that $P_1$ is
  reducible. Hence $F$ is reducible and we could have returned
  \False{} at this point. We did not take into account this obvious
  improvement in our algorithm for readibility.
\end{rem}

\subsection{Further comments}\label{ssec:comments}

\paragraph{Factorisation of quasi-irreducible polynomials.}
Not returning \False{} when discovering a factor of $Q$ also makes
sense if we want further informations about the factorisation of
$F$. Namely, if $F$ is quasi-irreducible, then we can deduce from the
field decomposition of $\Ki_g$ the number of irreducible factors of
$F$ in $\Ki[[x]][y]$ together with their residual degrees and index of
ramification. In Example \ref{xmp:splitting} above, we find the field
decomposition:
\[
  \Qi_2\simeq\frac{\Qi[Z_1,Z_2]}{(Z_1-1,Z_2^4+1)}\oplus
  \frac{\Qi[Z_1,Z_2]}{(Z_1+1,Z_2-1)}\oplus
  \frac{\Qi[Z_1,Z_2]}{(Z_1+1,Z_2^3+Z_2^2+Z_2-1)}.
\]
It follows that $F$ has three irreducible factors in $\Qi[[x]][y]$ of
respective residual degrees $4,1,3$ (which are given together with
their residue fields) and ramification index $q_1q_2=2$. In
particular, they have respective degrees $8,2,6$.

In fact, quasi-irreducible polynomials behave like irreducible
polynomials, in the sense that they are ``balanced'': all their
absolutely irreducible factors in $\algclos{\Ki}[[x]][y]$ have same
sets of characteristic exponents and same sets of pairwise
intersection multiplicities. These important data can be deduced from
the edge data, see \cite[Section 8]{PoWe19}; they characterise the
equisingular type of the germ of curve $(F,0)$, which coincides with
the topological equivalent class in the case $\Ki=\Ci$. Unfortunately,
$F$ might be balanced without being quasi-irreducible. In order to
characterise balanced polynomials, we need to modify slightly
Definition \ref{def:quasideg}, allowing several edges when
$q=1$. These aspects are considered in the longer preprint
\cite{PoWe19} and will be published in a forthcoming paper.

\paragraph{The case of non Weierstrass polynomials.} Up to minor
changes, we can use algorithm \PIrr{} to test the irreducibility of
any square-free polynomial $F\in \Ki[[x]][y]$, without assuming $F$
Weierstrass. If $\NP(F)$ is not straight, then $F$ is reducible. If
$\NP(F)$ is straight with positive slope, we replace $F$ by its
reciprocal polynomial. The leading coefficient is now invertible and
we are reduced to consider the case $F$ monic. Then, algorithm \PIrr{}
works exactly as in the Weierstrass case. However, the bound
$(\dy-1)\,f\leq\vF$ of Lemma \ref{lem:d-delta} does not hold
anymore. To get a similar complexity, we need to modify slightly the
algorithm: we do not compute primitive elements of $\Ki_k$ over the
field $\Ki$ but only over the next residue ring $\Ki_1=\Ki_{P_1}$. It
can be shown that the complexity becomes $\Ot(\vF+\dy)$. Moreover, we
eventually get a bivariate representation
$\Ki_g=\Ki[Z_1,Z](P_1(Z_1),Q(Z_1,Z))$ and checking that $\Ki_g$ is a
field requires now two univariate irreducibility tests of degree at
most $d$. See \cite[Section 7.4]{PoWe19} for details.

\paragraph{Bivariate polynomials.} If the input $F$ is given as a
bivariate polynomial $F\in \Ki[x,y]$ with partial degrees
$\dx:=\deg_x(F)$ and $\dy=\deg_y(F)$, we get a complexity estimate
$\Ot(\dx\dy)$ which is quasi-linear with respect to the arithmetic
size of the input. Moreover, we need not to assume $F$
square-free. Namely, we first reduce to the monic case as explained in
the previous paragraph. Then, we run algorithm \PIrr{} with parameters
$F$ and $4\,\dx$, except that we return \False{} whenever test of line
\ref{Pirr:callback} fails. If $F$ is square-free, we have the well
known inequality $\vF\le 2\dx \dy$ so that $\eta(F)\le 4\, \dx$: the
algorithm will return the correct answer with at most $\Ot(\dx\, \dy)$
operations over $\Ki$ as required, and so without reaching a value $\eta'>4\,\dx$ at Line \ref{Pirr:callback}. If $F$ is not square-free, then $\bar{H}_k$ is never
square-free. Hence, we will never reach the case $N_k=1$ and we end up with
three possibilities:
\begin{itemize}
\item we reach a value $\eta'>4\,\dx$ at Line \ref{Pirr:callback},
  ensuring the non square-freeness (hence the non quasi-irreducibility) of $F$;
\item the function call at Line \ref{Pirr:degen} returns \False{} and $F$ is not quasi-irreducible;
\item the function call at Line \ref{Pirr:degen} computes an edge data which satisfies $q=\deg(P)=1$ (this happens exactly when we compute an approximate root $\psi$ such that $F=\psi^N$ modulo $x^{4\,\dx+1}$ for some $N>1$). As this can not happen when $F$ is square-free, we deduce that $F$ is not square-free, hence not quasi-irreducible.
\end{itemize}
As we always truncate the powers of $x$ with precision $4\,\dx$, we
will return \False{} within an expected $\Ot(\dx \,\dy)$ operations
over $\Ki$ in all three cases. Note that in the second case, we can not conclude if $F$ is square-free or not. 

\paragraph{Absolute irreducibility.} We say that $F\in \Ki[[x]][y]$ is
absolutely irreducible if it is irreducible in
$\overline{\Ki}[[x]][y]$, that is if $F$ is quasi-irreducible and
$f_g=1$. To check this we can slightly modify algorithm \PIrr{}: just
return \False{} whenever $\ell_k>1$. We thus have
$\Ki_k=\Ki$ for all $k$, and do not need the Las-Vegas
subroutine \Primitive{}, nor any univariate irreducibility test. We
obtain a deterministic algorithm running with $\Ot(\vF+d)$ operations
over $\Ki$, which is $\Ot(\vF)$ is $F$ is Weierstrass. Also, we could
have used algorithm \AbhyankarMoh{} below with suitable precision for
the same cost.


\section{Abhyankhar's absolute irreducibility
  test}\label{sec:absolute}

Abhyankhar's absolute irreducibility test avoids any Newton-Puiseux
type transforms or Hensel type liftings. In fact, it is even stronger
as it does not require to compute the \edgepoly{}s
$\bar{H}_k$: knowing their Newton polygon is sufficient. Although we
do not need this improvement from a complexity point of view, we show
how to recover this result in our context for the sake of
completness. We will use the following alternative characterisations
of valuations and polygons. For convenience, we will rather compute the translated polygon $\NP_k(F):=\NP(H_k)+(0,v_k(F))$, which by \eqref{eq:pikHk} coincides with the union of edges of strictly negative slopes of $\NP(\pi_k^*(F))$. 

\begin{lem}\label{lem:samevk}
  Suppose that $H_0,\ldots,H_{k-1}$ are degenerated.
  \begin{enumerate}
  \item\label{enum:NkF} Write $F=\sum c_i\psi_k^i$ the $\psi_k$-adic
    expansion of $F$. Then $\val[k](F)=\min_i \val[k](c_i\psi_k^i)$
    and
    \begin{equation}\label{defpolygk}
      \NP_k(F)=\Conv\left(\left(i,\val[k](c_i\psi_k^i)\right)+(\Ri^+)^2, \, c_i\ne 0\right).
    \end{equation}
  \item\label{enum:vk} Let $k\ge 1$ and $G\in \Ki[[x]][y]$ with $\psi_{k-1}$-adic
    expansion $G=\sum_i a_i\psi_{k-1}^i$. We have
    \begin{equation}\label{newdefvk}
      \val[k](G)=\min_i \left(q_k v_{k-1}(a_i\psi_{k-1}^i)+i m_k \right).
    \end{equation}
  \end{enumerate}
\end{lem}

\begin{proof}
  Equality \eqref{defpolygk} is a direct consequence of Corollary
  \ref{cor:gprime} with Theorems \ref{thm:NPphi} and
  \ref{thm:HPsiPhi}. Also, from \eqref{eq:ajpsij},
  $\pi_k^*(c_i\psi_k^{i})$ has a term of lowest $x$-valuati
on of shape
  $ux^{\val[k](a_i\,\psi_k^i)}\,y^i$ for some $u\in \Ki_k^\times$ and
  it follows that $\val[k](F)=\min_i \val[k](c_i\psi_k^i)$. This
  proves Point \ref{enum:NkF}.

  Applying \eqref{eq:ajpsij} at rank $k-1$, we get
  $\pi_{k-1}^*(a_i\psi_{k-1}^i)=x^{v_{k-1}(a_i\psi_{k-1}^i)}\,(y+x^{\alpha}\tilde{U}_i)^i\,U_i$,
  where $\alpha> m_k/q_k$, and $U_i,\tilde{U}_i$ are units. As
  $m_k>0$, we deduce that
  $V_i=U_i(z_k^{s_k}x^{q_k},x^{m_k}(y+z_k^{t_k}+c_k(x))$ is a unit
  such that $V_i(0,y)=U_i(0,0)\in \Ki_k^\times$ is constant and a
  straightforward computation shows
  $\pi_{k}^*(a_i\psi_{k-1}^i)=x^{q_k v_{k-1}(a_i\psi_{k-1}^i)+i
    m_k}P_i(y)+h.o.t$, where $P_i\in \Ki[y]$ has degree exactly
  $i$. Equality \eqref{newdefvk} follows.
\end{proof}

\begin{rem}\label{rem:refMontes}
  Point \ref{enum:vk} in Lemma \ref{lem:samevk} shows that our valuations coincide
  with the extended valuations used in the Montes algorithm over general
  local fields; see for instance \cite[Point (3) of Proposition
  2.7]{GuMoNa12}.
\end{rem}

Hence, we may take \eqref{defpolygk} and \eqref{newdefvk} as
alternative recursive definitions of valuations and Newton
polygons. This new point of view has the great advantage to be
independent of the map $\pi_k$, hence of the Newton-Puiseux
algorithm. In particular, it can be generalised at rank $k+1$ without
assuming that $H_k$ is degenerated.

\begin{dfn}\label{def:alternative}
  Suppose that $H_0,\ldots,H_{k-1}$ are degenerated and let
  $-m_{k+1}/q_{k+1}$ be the slope of the lowest edge of $H_k$. We
  still define the valuation $v_{k+1}$ and the Newton polygon
  $\NP_{k+1}(F)$ by formulas \eqref{newdefvk} and \eqref{defpolygk}
  applied at rank $k+1$.
\end{dfn}

\begin{rem}
This definition of the map $v_{k+1}$ is equivalent to
\[
  v_{k+1}(G)=\min_{g_B\ne 0} \left(q_{k+1}\langle B, V \rangle
    +m_{k+1} b_k\right)
\]
where $G$ has $(\psi_{-1},\ldots,\psi_k)$-adic expansion
$G=\sum g_B\Psi^B$ and $V=(v_{k,-1},\ldots, v_{k,k})$. This is the
approach we shall use in practice to update valuations.
\end{rem}

We obtain the following absolute irreducibility test which only
depends on the geometry of the successive Newton polygons.

\begin{algorithm}[H]
  \nonl\TitleOfAlgo{\AbhyankarMoh($F$)\label{algo:AbMoh}}%
  \KwIn{$F\in\Ki[[x]][y]$ Weierstrass s.t. \Char($\Ki$) does not
    divide $\dy=\deg(F)$.}%
  \KwOut{\True{} if $F$ is irreducible in $\overline{\Ki}[[x]][y]$,
    \False{} otherwise.}%
  $N\gets \dy$, $v_0\gets\val$, $k\gets0$\;%
  \While{$N > 1$}{ $\psi\gets \AppRoot{}(F,N)$\;%
    $\sum c_i\psi^i \gets \Expand{}(F,\psi)$\;%
    Compute $\NP_k(F)$ using \eqref{defpolygk}\;%
    \If{$(N,v_k(F))\notin \NP_k(F)$ or $\NP_k(F)$ is not straight or
      $q=1$}{%
      \Return{\False}%
    }%
    $N\gets N/q$, $k\gets k+1$\;%
    Compute $v_k$ from $v_{k-1}$ via \eqref{newdefvk}\;%
  }%
  \Return{\True}\;%
\end{algorithm}

\begin{prop}\label{prop:abhyankhar}
  Algorithm \AbhyankarMoh{} works as specified.
\end{prop}

\begin{proof}
  We need to show that it returns the same output as
  \irreducible($F,\overline{\Ki}$).  Suppose that $F$ is not
  absolutely irreducible. Let us abusively still denote by $g$ be the
  first index $k$ such that $H_k$ is not degenerated over
  $\overline{\Ki}$ or $N_k=1$: so both algorithms \AbhyankarMoh{}($F$)
  and \irreducible ($F,\overline{\Ki}$) compute the same data
  $\psi_0,\ldots,\psi_{g-1}$ and $(q_1,N_1),\ldots,(q_g,N_g)$. If
  $N_g=1$, then $F$ is absolutely irreducible, and both algorithms
  return \True{} as required. If $N_g>1$, then \irreducible
  ($F,\overline{\Ki}$) returns \False. As  $\NP_g(F)=\NP(H_g)+(0,v_g(F))$ (definition) and $H_g$ is
  Weierstrass of degree $N_g$, we have $(N_g,v_g(F))\in \NP_g(F)$ at
  this stage. If $\NP_g(F)$ is not straight or $q_{g+1}=1$, then so
  does $\NP(H_g)$ and \AbhyankarMoh{}($F$) returns \False{} as
  required. There remains to treat the case where $\NP_g(F)$ is
  straight with $q_{g+1}>1$ (still assuming $N_g>1$ and $H_g$ not
  degenerated over $\overline{\Ki}$). In such a case,
  \AbhyankarMoh{}($F$) computes the next $N_{g+1}^{th}$ approximate
  roots $\psi_{g+1}$ of $F$ where $N_{g+1}=N_g/q_{g+1}$. We will show
  that $(N_{g+1},v_{g+1}(F))\notin \NP_{g+1}(F)$ so that
  \AbhyankarMoh{} returns \False{} at this step.

  Let $F=\sum_{i=0}^{N_{g+1}} c_i \psi_{g+1}^{i}$ be the
  $\psi_{g+1}$-adic expansion of $F$. By hypothesis, we know that
  \[
    \pi_{g}^*(F)=x^{v_g(F)}\, H_g\, U,\text{ with } U(0,0)\ne 0
  \]
  where $\bar{H}_g=\prod_{Q(\zeta)=0}(y^{q_{g+1}}-\zeta x^{m_{g+1}})$,
  with $Q\in \Ki[Z]$ of degree $N_{g+1}:=N_{g}/q_{g+1}$ having at
  least two distinct roots. In particular, $\bar{H}_g$ is not the
  $N_{g+1}$-power of a polynomial and it follows that
  $\pi_g^*(\psi_{g+1}^{N_{g+1}})$ and $\pi_g^*(F)$ can not have the
  same \edgepoly{}s. We deduce that there is at least one index
  $i<N_{g+1}$ such that $\NP_g(c_i \psi_{g+1}^{i})$ has a point on or
  below $\NP_g(F)$. Consider the $\psi_g$-adic expansions
  $c_i \psi_{g+1}^{i}=\sum_{j} a_{j} \psi_g^j$ and
  $F=\sum_j \alpha_j\psi_g^j$. Thanks to \eqref{defpolygk}, there
  exists at least one index $j$ such that
  $(j,v_g(a_{j} \psi_g^j))\in \NP_g(c_i \psi_{g+1}^i)$. By
  \eqref{defpolygk}, $\NP_g(F)$ is the lower convex hull of
  $(j,v_g(\alpha_{j} \psi_g^j))$, which is by assumption straight of
  slope $-q_{g+1}/m_{g+1}$. It follows that
  \[
    \min_{j} (q_{g+1} v_g(a_j\psi_g^j)+m_{g+1} j)\le \min_j (q_{g+1}
    v_g(\alpha_j\psi_g^j)+m_{g+1} j).
  \]
  Thanks to Definition \ref{def:alternative}, this implies
  $v_{g+1}(c_i \psi_{g+1}^{i})\le v_{g+1}(F)$ which in turns forces
  $(N_{g+1},v_{g+1}(F))\notin \NP_{g+1}(F)$.
\end{proof}


\bibliographystyle{abbrv} {\bibliography{tout}}
\addcontentsline{toc}{section}{References.}

\end{document}